\documentclass[final,onefignum,onetabnum]{siamart171218} % SIAM
\newif\ifsiamart
\makeatletter%
\@ifclassloaded{siamart171218}%
  {\siamarttrue}%
  {\siamartfalse}%
\makeatother%

\ifsiamart\else
\usepackage{authblk}
\fi
\usepackage{silence}
\usepackage[T1]{fontenc}
\usepackage[utf8]{inputenc}
\usepackage{csquotes}
\usepackage[UKenglish]{babel}
\usepackage{paralist}
\usepackage[shortlabels]{enumitem}
\usepackage{xcolor}
\usepackage{graphicx}
\usepackage{bbm}
\usepackage{amsmath}
\usepackage{amssymb}
\usepackage{amsfonts}
\usepackage{stmaryrd}
\usepackage{mathrsfs}
\usepackage{mathtools}
\usepackage{epstopdf}
\usepackage{comment}
\usepackage{subcaption}
\usepackage{xspace}
\usepackage{float}
\usepackage{stackengine}

\ifsiamart\else
\usepackage{amsthm}
\usepackage{hyperref}
\usepackage[margin=.6in]{geometry}
\usepackage{setspace}
\hypersetup{%
    linktocpage=true,
    colorlinks=true,
    citecolor=red,
    linkcolor=blue,
    urlcolor=blue,
}
\fi
\usepackage{amsopn}

\newcommand{\norm}[1]{\left\lVert#1\right\rVert}
\DeclarePairedDelimiter\myip{\langle}{\rangle}

\DeclareMathOperator{\e}{e}

\DeclareMathOperator*{\argmin}{arg\,min}

\renewcommand{\d}{\mathrm d}

\newcommand{\N}{\mathbb{N}}
\newcommand{\R}{\mathbb R}
\newcommand{\real}{\R}
\newcommand{\T}{\mathbb T}
\newcommand{\torus}{\T}
\newcommand{\nat}{\N}
\newcommand{\proba}{\mathbb P}
\newcommand{\expect}{\mathbb{E}}
\newcommand{\var}{\mathbb{V}}
\newcommand{\ee}{\mathrm e}

\renewcommand{\leq}{\leqslant}
\renewcommand{\geq}{\geqslant}
\renewcommand{\le}{\leqslant}
\renewcommand{\ge}{\geqslant}
\newcommand{\cbdg}[1]{C_{\rm{BDG}}^{(#1)}}

\ifsiamart
\newtheorem{remark}[theorem]{Remark}
\newtheorem{assumption}{Assumption}
\else
\theoremstyle{plain}

\newtheorem{remark}[theorem]{Remark}

\newtheorem{assumption}{Assumption}
\fi

\crefname{lemma}{Lemma}{Lemmas}
\crefname{remark}{Remark}{Remarks}
\crefname{assumption}{Assumption}{Assumptions}
\crefname{proposition}{Proposition}{Propositions}
\crefname{section}{Section}{Sections}
\crefname{subsection}{Subsection}{Subsections}
\crefname{equation}{}{}
\Crefname{equation}{Equation}{Equations}
\newlist{lemmaenum}{enumerate}{3}
\setlist[lemmaenum]{label=(\alph*),ref=\,(\alph*)}
\crefname{lemmaenum}{Lemma}{Lemmas}
\newlist{assumpenum}{enumerate}{5}
\setlist[assumpenum]{label=(\alph*), font={\bfseries}}
\newlist{auxenum}{enumerate}{2}
\setlist[auxenum]{label=(\alph*),ref=(\alph*)}
\crefname{auxenumi}{Item}{Items}
\crefname{enumi}{}{}
\crefname{equation}{}{}
\crefname{assumpenumi}{}{}
\crefname{assumpenumii}{}{}
\Crefname{assumpenumi}{Assumption}{Assumptions}
\Crefname{assumpenumii}{Assumption}{Assumptions}
\Crefname{assumpenumii}{Assumption}{Assumptions}
\crefrangeformat{assumpenumi}{#3#1#4--#5#2#6}
\Crefname{lemmaenumi}{Part}{Parts}
\Crefname{figure}{Figure}{Figures}
\numberwithin{equation}{section}
\numberwithin{theorem}{section}

\usepackage{tikz}
\usetikzlibrary{shapes.misc}
\usetikzlibrary{positioning}

\ifsiamart\else
\let\oldparagraph=\paragraph
\renewcommand\paragraph[1]{\oldparagraph{#1.}}
\usepackage[url=true,backend=biber,style=trad-abbrv,url=false,doi=false,isbn=false]{biblatex}
\DeclareFieldFormat{volume}{volume \textbf{#1}}
\DeclareFieldFormat[article]{volume}{\textbf{#1}}

\fi

\newcommand{\qapprox}[1]{\hat{q}_{#1}}
\newcommand{\qapproxa}[1]{\hat{q}_{#1}^\alpha}
\definecolor{darkred}{rgb}{.7,0,0}
\definecolor{darkgreen}{rgb}{.1,.7,0}

\graphicspath{figures/}
\ifpdf
  \DeclareGraphicsExtensions{.eps,.pdf,.png,.jpg}
\else
  \DeclareGraphicsExtensions{.eps}
\fi

\title{Dynamical reweighting for estimation of fluctuation formulas}

\ifsiamart
    \author{%
        Raphaël Gastaldello\thanks{%
        CNRS 3522 Bézout, France \& CERMICS, ENPC, Institut Polytechnique de Paris, Marne-la-Vallée, France \& MATHERIALS, Inria Paris, France
            (\email{raphael.gastaldello@enpc.fr}).
        } \and
        Gabriel Stoltz\thanks{%
        CERMICS, ENPC, Institut Polytechnique de Paris, Marne-la-Vallée, France \& MATHERIALS, Inria Paris, France (\email{gabriel.stoltz@enpc.fr}).
        } \and
        Urbain Vaes\thanks{%
            MATHERIALS, Inria Paris, France \& CERMICS, ENPC, Institut Polytechnique de Paris, Marne-la-Vallée, France (\email{urbain.vaes@inria.fr})
        }
    }
\else
    \author[1,2]{R. Gastaldello$^{a,}$}
    \author[1,2]{G. Stoltz$^{b,}$}
    \author[2,1]{U. Vaes $^{c,}$}
    \affil[ ]{\footnotesize $^a$\email{raphael.gastaldello@enpc.fr},
    $^b$\email{gabriel.stoltz@enpc.fr},
    $^c$\email{urbain.vaes@inria.fr}}
    \affil[1]{\footnotesize CERMICS, \'Ecole nationale des Ponts et Chaussées, France}
    \affil[2]{\footnotesize MATHERIALS project-team, Inria Paris, France}
    \date{\today}
\fi

\ifpdf
\hypersetup{
  pdftitle={},
  pdfauthor={R. Gastaldello, G. Stoltz, and U. Vaes}
}
\fi
\headers{Dynamical reweighting for estimation of fluctuation formulas}{R. Gastaldello, G. Stoltz, and U. Vaes}
\begin{document}
\maketitle
\begin{abstract}
    We propose a variance reduction method for calculating transport coefficients in molecular dynamics using an importance sampling method via Girsanov's theorem applied to Green--Kubo's formula. We optimize the magnitude of the perturbation applied to the reference dynamics by means of a scalar parameter~$\alpha$ and propose an asymptotic analysis to fully characterize the long-time behavior in order to evaluate the possible variance reduction. Theoretical results corroborated by numerical results show that this method allows for some reduction in variance, although rather modest in most situations.
\end{abstract}



\section{Introduction}%
\label{sec:Introduction}
Molecular dynamics has been developing rapidly for over 70 years~\cite{battimelli2020computer}. It relies on statistical physics, which offers a path to understand the macroscopic behavior of matter from its microscopic description. Despite severe constraints in terms of time and space scales, it has proved to be an extremely powerful tool for carrying out \emph{in silico} experiments, testing physical theories on a computer and also helping to make up for the absence of analytical theories. Nowadays, numerical simulations are essential for understanding the macroscopic behaviour of matter, particularly when experiments are impossible or costly, such as in high-pressure or high-temperature regimes, or when studying as yet unknown or unsynthesised materials especially for electronic and biological systems on the nanometric scale. Another major objective of molecular simulation is to calculate macroscopic quantities and thermodynamic properties, providing quantitative information about molecular systems. We refer for instance to~\cite{allen2017computer,frenkel2023understanding,tuckerman2023statistical} for a large panel of numerical methods in statistical physics and molecular simulation and~\cite{leimkuhler2015molecular,stoltz2010free} for a mathematical perspective.

One of the aims of molecular dynamics is to understand the dynamical properties of matter. Transport coefficients such as those governing particle diffusion (self-diffusion), heat conduction (thermal conductivity), or momentum transport (shear viscosity), are examples of such properties. They determine the response of a system to an external perturbation; see~\cite{j2007statistical,todd2017nonequilibrium} for an general overview and~\cite[Section~5]{LS} for a mathematical perspective on this subject. Once these coefficients have been estimated through molecular simulation, they can be fed to macroscopic evolution equations like the Navier--Stokes equation.

A standard method for computing transport coefficients, based on molecular dynamics, utilizes the well-known Green–Kubo formula~\cite{green1952markoff,kubo1957statistical}. This formula expresses transport coefficients as integrated time-correlations of relevant observables for the system at equilibrium. Another approach, relying on linear response theory, involves directly simulating non-equilibrium perturbations and measuring the resulting average response, a method known as non-equilibrium molecular dynamics (NEMD).

However, estimators derived from these methods often suffer from significant statistical errors, as discussed in~\cite{Sto1}. Achieving convergence often demands simulating long trajectories, incurring substantial computational costs. Various variance reduction techniques can be used for systems at equilibrium, in particular importance sampling~\cite{chak2023optimalimportancesamplingoverdamped} and control variates methods~\cite{mira2013zero,oates2017control}. Most variance reduction techniques cannot be used as such for nonequilibrium systems, and therefore require a dedicated treatment, see for instance~\cite{roussel2019,Renato&Gab,blassel2024fixing,pavliotis2024neuralnetworkapproachesvariance}.

Concerning importance sampling, one has to resort to dynamical versions where a bias is introduced into the dynamics, and corrected for by some exponential weight based on Girsanov's formula. This is known as ``dynamical reweighting'' or ``Girsanov reweightings'' in the theoretical chemistry literature. For example, in~\cite{chodera2011dynamical} this method is used to compute time-correlation functions, and the references~\cite{DHK,donati2018girsanov,donati2022review} present how to implement the reweighting and use it to simulate proteins, or accelerate the exploration of the phase space. In all cases, the aim is to improve the computation of path ensemble averages.

\paragraph{Objectives of this work}%
The present work aims to study and assess the efficiency of the particular method described in~\cite{DHK,donati2018girsanov}
when used as a variance reduction technique for the calculation of transport coefficients.
Mathematically, these coefficients are expressed as integrated (auto)correlations in the Green--Kubo formalism
and can therefore be reformulated as ensemble averages with respect to the path measure associated with the stochastic dynamics.
In the context of non-degenerate diffusive dynamics,
the key idea of the method we study is to add a drift term in order to bias the dynamics:
\begin{align}
    \label{eq:biased}
    \d q_t^\alpha = b(q_t^\alpha)\,\d t + \alpha \sigma(q_t^\alpha)u(q_t^\alpha)\,\d t + \sigma(q_t^\alpha)\,\d W_t.
\end{align}
This modified dynamics induces a path measure that is different from that associated with the reference dynamics~($\alpha = 0$).
Therefore, for consistency,
statistical estimators based on the modified dynamics~\eqref{eq:biased} must include
Girsanov weights,
which represent the likelihood ratio between the reference and modified path probability measures.
In this work, we wish to understand whether
employing a perturbed dynamics of the type~\eqref{eq:biased} instead of the reference dynamics
can be useful for the calculation of transport coefficients through Green–Kubo’s formula.
\paragraph{Our contributions}
The paper makes the following contributions.
\begin{itemize}
    \item
        We first consider the following stochastic differential equation (SDE),
        \begin{align}
            \label{eq:reference-dynamics}
            \d q_t = b(q_t)\,\d t + \sigma(q_t)\,\d W_t.
        \end{align}
        In order to motivate this study,
        we prove for this reference dynamics that the variance of the Green--Kubo estimator is proportional to the integration time~$T$,
        in the limit as~$T$ tends to infinity (see~\cref{proposition:asymptotic_variance_gk}).
    \item
        We then consider the dynamics~\eqref{eq:biased}
        incorporating an additional drift term $\alpha \sigma u$ compared to~\eqref{eq:reference-dynamics},
        with~$\alpha$ a scalar and $u$ a vector field.
        After defining a weighted Green--Kubo-type estimator in terms of the solution to this equation,
        we show that, for any fixed time $T$,
        there exists a unique $\alpha$ that minimizes the variance of the weighted Green--Kubo's estimator (see~\cref{prop:differentiation}).
    \item
        Finally, we characterize the asymptotic behavior,
        as the integration time~$T$ tends to $+\infty$,
        of the optimal value of $\alpha$ and of the amount of variance reduction obtained with this asymptotically optimal value.
\end{itemize}
We emphasize that the main contribution of this work is the asymptotic analysis of the new estimator for the Green--Kubo formula. We however also explore the non-asymptotic regime through numerical simulations. The so-obtained results are quantitatively similar to those predicted by the asymptotic analysis, therefore achieving only a modest variance reduction. The development of numerical methods based on Girsanov's theorem seems to be more effective for the Markov state models considered in~\cite{DHK}.

\paragraph{Outline of the paper}
In~\cref{sec:replin-and-GK},
we present the general mathematical context of the computation of transport coefficients through the Green--Kubo estimator,
and study the scaling of the variance of this estimator with respect to the integration time.
In~\cref{sec:variance-reduction},
we present and study the dynamical reweighting approach.
Specifically,
we apply Girsanov's theorem in order to construct a new Green--Kubo-like estimator based on the biased dynamics~\eqref{eq:biased} and incorporating Girsanov weights,
depending on the parameter $\alpha$.
We show that there exists a unique value of this parameter which minimizes the estimator variance,
and give its scaling with respect to the integration time. 
In~\cref{sec:numerical-results} we illustrate our main results,
in particular those concerning the amount of variance reduction,
with numerical experiments.
Finally, \cref{sec:conclusion} is reserved for discussions and perspectives for future work.
\section{Mathematical framework}
\label{sec:replin-and-GK}
We first present in~\cref{sec:ref_dyn} the general framework of our study, including a reference stochastic differential equation (SDE)
to describe the evolution of the system. We then present in~\cref{sec:transport_coefficients} the celebrated Green--Kubo formula for the computation of transport coefficients,
and provide a novel asymptotic quantification of the variance of the associated estimator.
\subsection{General setting}
\label{sec:ref_dyn}
We first present the general setting of our study and then specify the analysis to the overdamped Langevin dynamics.
We consider a time-homogeneous SDE defined on the state-space $\mathcal{X}$, where $\mathcal{X}$ is typically either the entire space $\real^d$ or a bounded domain with periodic boundary conditions $\torus^d$ (with $\torus = \real / \mathbb{Z}$ representing the one-dimensional torus):
\begin{align}
    \label{eq:dyn_brow}
    \d q_t = b(q_t)\,\d t + \sigma(q_t)\,\d W_t.
\end{align}
Here
$W_t \in  \real^m$ is a standard Brownian motion,
and the functions $b \colon\mathcal{X} \mapsto \real^d$ and $\sigma \colon\mathcal{X} \mapsto \real^{d\times m}$
are the drift and diffusion coefficients.
We introduce the following assumption for simplicity.
\begin{assumption}
    \label{assumption:ex_mes_inv}
    The state space~$\mathcal{X}$ is the d-dimensional torus~$\torus^d$, the functions~$b\colon \real^d \to \real^d$ and~$\sigma\colon \real^d \to \real^{d \times m}$ are~$C^\infty$, the diffusion matrix~$\sigma \sigma^\top$ is positive definite and uniformly bounded from below over $\torus^d$.
    Consequently, the dynamics~\eqref{eq:dyn_brow} admits a unique invariant probability measure~$\mu$.
\end{assumption}
\begin{remark}
    The existence of an invariant probability measure mentioned in~\cref{assumption:ex_mes_inv} is a consequence of the non-degeneracy of the matrix~$\sigma \sigma^\top$,
    which is reflected in the ellipticity of the infinitesimal generator associated to~\eqref{eq:dyn_brow}.
    A variety of techniques can be employed to show this,
    such as as those based on the classical Doeblin minorization condition.
    See~\cite[Theorem 6.16]{pavliotis2008multiscale} for an ergodicity result that applies directly to our setting, 
    or~\cite{arnold1987unique} for a more general result which also applies to degenerate diffusions.
\end{remark}

Note that \cref{assumption:ex_mes_inv} guarantees the existence and uniqueness of a strong solution of~\eqref{eq:dyn_brow}.
In the sequel, $\langle\cdot,\cdot\rangle$ and $\|\cdot\|$ refer respectively to the inner product and its associated norm in~$L^2(\mu)$ for the  probability measure $\mu$ from~\cref{assumption:ex_mes_inv},
namely
\[
    \forall f,g \in L^2(\mu),\qquad \langle f,g\rangle = \int_{\torus^d} fg\,\d \mu.
\]
We also introduce the projection operator
\begin{align}
    \Pi \varphi = \varphi - \expect_{\mu}[\varphi],
\end{align}
where $\expect_{\mu}[\varphi]$ is the expectation of a function with respect to $\mu$. We denote the Hilbert space
\[
    L^2_0(\mu) = \left\{f \in L^2(\mu) \ \middle| \ \int_{\torus^d} f\,\d\mu = 0 \right\} = \Pi L^2(\mu),
\]
which is the subspace of functions of $L^2(\mu)$ with zero expectation with respect to~$\mu$.
The infinitesimal generator associated to the SDE \eqref{eq:dyn_brow} is given by the following differential operator:
\[
    \mathcal{L} = b^\top \nabla + \frac12 \sigma\sigma^\top : \nabla^2,
\]
where $:$ denotes the Frobenius inner product, and $\nabla ^2$ is the Hessian operator. More explicitly, for a given $C^\infty$ test function $\phi : \real^d \rightarrow \real$, the operator $\mathcal{L}$ acts as
\[
    \mathcal{L}\phi = \sum_{i = 1}^{d}b_i \frac{\partial\phi}{\partial x_i} + \frac12\sum_{i,j = 1}^{d}\sum_{k = 1}^{m}\sigma_{i,k}\sigma_{j,k}\frac{\partial^2\phi}{\partial x_i\partial x_j}.
\]
The Markov semigroup of operators $t \mapsto \operatorname{e}^{t\mathcal{L}}$ associated with~\eqref{eq:dyn_brow},
with generator~$\mathcal L$,
gives the time evolution of expectation of functions~$\varphi\colon \torus^d \to \real$,
henceforth called observables, along the dynamics:
\begin{align}
    \label{eq:def_semigroup}
    \left(\operatorname{e}^{t\mathcal{L}}\varphi\right)(q_0) = \expect^{q_0}\left[\varphi(q_t)\right].
\end{align}
Here, 
the expectation is over all realizations of $W_t$ in \eqref{eq:dyn_brow} for a dynamics starting at a deterministic initial condition~$q_0$.
We now state an assumption on the semigroup that will be essential in the sequel.
It implies in particular that the operator $\mathcal{L}$ is invertible on $L^2_0(\mu)$.
\newcommand{\cdecay}{L}
\begin{assumption}
    \label{assumption:inv_decay}
    Suppose that there exist $L \in \mathbb{R}_+$ and $\lambda > 0$ such that
    \[
        \forall t \ge 0,\qquad \|\mathrm{e}^{t\mathcal{L}}\|_{\mathcal{B}(L^2_0(\mu))}\le \cdecay\mathrm{e}^{-\lambda t}.
    \]
    As a consequence the operator $\mathcal{L}$ is invertible on $L^2_0(\mu)$, with
    \[
        \mathcal{L}^{-1} = -\int_0^{+\infty} \mathrm{e}^{t\mathcal{L}}dt,
    \]
    and
    \[
        \| \mathcal{L}^{-1} \|_{\mathcal{B}(L^2_0(\mu))} \le \frac L \lambda.
    \]
\end{assumption}
A paradigmatic example of a dynamics for which \cref{assumption:ex_mes_inv,assumption:inv_decay} are satisfied is the overdamped Langevin dynamics
\begin{align}
    \label{ovdL}
    \d q_t = -\nabla V (q_t)\,\d t + \sqrt{\frac{2}{\beta}}\,\d W_t,
\end{align}
where $\beta > 0$ is proportional to the inverse temperature, and $V$ is a smooth potential.
The dynamics~\eqref{ovdL} admits the Gibbs probability measure with density
\[
    \mu(q) = \frac1Z \ee^{-\beta V(q)},\qquad Z = \int_\mathcal{X} \ee^{-\beta V} < +\infty,
\]
as its unique invariant probability measure; see \cite{bakry2014analysis} for various results on this dynamics.
\subsection{Computation of transport coefficients and Green--Kubo formula}
\label{sec:transport_coefficients}
Transport coefficients provide insights into how a system responds to perturbations from equilibrium.
One common method for computing them using molecular dynamics involves the famous Green–Kubo formula.
This formula expresses transport coefficients as integrated time-correlations between relevant fluxes within the equilibrium system.
The general form of the Green--Kubo formula is the following:
\begin{align}
\label{green-kubo_coeff}
\rho= \int_0^{+\infty} \mathbb{E }_\mu \left[R(q_t)S(q_0)\right]\d t,
\end{align}
where $R$, $S \in L^2_0(\mu)$ and $\mu$ is the invariant probability measure of \eqref{eq:dyn_brow}.
A natural estimator of~\eqref{green-kubo_coeff} is obtained by first truncating the time integral to some finite time~$T$ and then replacing the expectation by the empirical average over $J$ independent realizations of the equilibrium dynamics~\eqref{ovdL}
with stationary initial conditions~$q^{(j)}_0 \sim \mu$.
This leads to the estimator
\begin{equation}
    \label{eq:estimator_continuous_time}
    \widehat{\rho}_{T,J} = \frac{1}{J}\sum_{j=1}^{J}\int_0^T R\Bigl(q_t^{(j)}\Bigr)S \Bigl(q_0^{(j)}\Bigr) \, \d t.
\end{equation}
The truncation bias $\left\lvert  \expect_{\mu}\left[\widehat{\rho}_{T,J}\right] - \rho\right\rvert$ is exponentially small in the integration time~$T$,
in view of the exponential decay of correlations from~\cref{assumption:inv_decay}. Indeed,
\begin{align*}
    \bigl\lvert  \expect_{\mu}\left[\widehat{\rho}_{T,J}\right] - \rho \bigr\rvert
    = \left\lvert \int_T^{+\infty} \expect_\mu \left[R(q_t)S(q_0)\right]\d t \right\rvert \leq \int_T^{+\infty} \left|\left\langle\ee^{t\mathcal{L}}R,S\right\rangle\right| \d t,
\end{align*}
so that, by the Cauchy--Schwarz inequality,
\begin{align}
\label{bias_time_trunc}
\left\lvert  \expect_{\mu}\left[\widehat{\rho}_{T,J}\right] - \rho\right\rvert \leq \|R\|\|S\|\int_T^{+\infty}L\ee^{-\lambda t}\d t = \frac{\|R\|\|S\|}{\lambda} L\ee^{-\lambda T}.
\end{align}
The above inequality initially suggests that $T$ should be sufficiently large to meet a small error threshold.
However, the estimator $\widehat{\rho}_{T,J}$ of the Green--Kubo formula has a variance that increases linearly with time $T$,
as made precise in the following proposition.
Note that the next result was proven in a similar way in~\cite{pavliotis2024neuralnetworkapproachesvariance} at the same this work was completed.
\begin{proposition}
    \label{proposition:asymptotic_variance_gk}
    Let $R, S \in L^2_0(\mu)$ be two smooth observables. Under~\cref{assumption:inv_decay}, the variance of the estimator $\widehat{\rho}_{T,J}$ given in~\eqref{eq:estimator_continuous_time} scales asymptotically as $T$:
    \begin{align}
        \label{eq:scaling_var_gk}
        \lim\limits_{\substack{T \to \infty}} \frac1T \var \left( \widehat{\rho}_{T,J} \right) = \frac{2}{J} \|S\|^2\left\langle R, -\mathcal{L}^{-1}R \right\rangle.
    \end{align}
\end{proposition}
\begin{proof}
    Since $\var(\widehat{\rho}_{T,J}) = J^{-1}\var(\widehat{\rho}_{T,1})$,
    it suffices to prove the result for $J=1$.
    Denote by $\mathcal{R} \in L^2_0(\mu)$ the solution to the Poisson equation $-\mathcal L \mathcal R = R$. The latter equation is well posed in view of~\cref{assumption:inv_decay}.
    By It\^o's formula,
    it holds
    \[
        \frac{1}{\sqrt{T}} \int_0^T R(q_t) S(q_0)\, \d t = \frac{S(q_0)}{\sqrt{T}}  \bigl( \mathcal R(q_0) - \mathcal R(q_T) \bigr) +
        \frac{S(q_0)}{\sqrt{T}} \int_{0}^{T} \nabla \mathcal R(q_t)^\top \sigma(q_t) \, \d W_t.
    \]
    The function $S$ is bounded uniformly over~$\torus^d$,
    and so is~$\mathcal R$ by boundedness of~$R$ and elliptic regularity.
    Therefore, the first term on the left-hand side of the above equality tends to 0 in squared expectation in the limit~$T \to +\infty$.
    Thus,
    using It\^o's isometry for the second term,
    we obtain
    \begin{align*}
        &\lim_{T \to \infty} \frac{1}{T} \expect_{\mu}\!\left[ \left( \int_0^T R(q_t) S(q_0) \, \d t \right)^2 \right]
        = \lim_{T \to \infty} \frac{1}{T} \int_{0}^{T} \expect_{\mu} \Bigl[ \bigl\lvert S(q_0) \bigr\rvert^2 |\sigma^\top \nabla \mathcal R(q_t)|^2 \Bigr]  \, \d t \\
        & \qquad\qquad\qquad\qquad\qquad  = \norm{S}^2 \expect_{\mu} \Bigl[ \Gamma_{\mathcal R} \Bigr]
        + \lim_{T \to \infty} \frac{1}{T} \int_{0}^{T} \expect_{\mu} \Bigl[ \bigl\lvert S(q_0) \bigr\rvert^2 \Pi \Gamma_{\mathcal R}(q_t) \Bigr]  \, \d t,
    \end{align*}
    where we used stationarity and introduced $\Gamma_{\mathcal R} := |\sigma^\top \nabla \mathcal R(q_t)|^2 = \nabla \mathcal R^\top \sigma\sigma^\top \nabla \mathcal R$. \\
    By conditioning on~$q_0$,
    then using the tower property of conditional expectation
    and the definition of the semigroup~\eqref{eq:def_semigroup},
    we obtain by~\cref{assumption:inv_decay},
    \begin{align*}
        \expect_{\mu} \Bigl[ \bigl\lvert S(q_0) \bigr\rvert^2 \Pi \Gamma_{\mathcal R}(q_t) \Bigr] 
        &= \expect_{\mu} \Bigl[ \expect\!\left[\bigl\lvert S(q_0) \bigr\rvert^2 \Pi \Gamma_{\mathcal R}(q_t) \,\middle|\, q_0\right]\Bigr] \\
        & = \expect_{\mu} \Bigl[ \bigl\lvert S(q_0) \bigr\rvert^2 \expect^{q_0}\!\left[ \Pi \Gamma_{\mathcal R}(q_t) \right]\Bigr] \\
        & = \expect_{\mu} \Bigl[ \bigl\lvert S(q_0) \bigr\rvert^2  \ee^{t\mathcal{L}}\Pi \Gamma_{\mathcal R}(q_0)\Bigr].
    \end{align*}
    From the last expression, 
    by the Cauchy--Schwarz inequality and smoothness of the observable ${S\colon \torus^d \to \real^d}$ we have
    \begin{align*}
        &\left|\frac{1}{T} \int_{0}^{T} \expect_{\mu} \Bigl[ \bigl\lvert S(q_0) \bigr\rvert^2 \Pi \Gamma_{\mathcal R}(q_t) \Bigr]  \, \d t\right| 
        \leq \frac 1 T\norm{S}_{L^4(\mu)}^2\int_{0}^{T} \norm{\ee^{t\mathcal{L}}\Pi \Gamma_{\mathcal R}}\, \d t \\
        &\qquad\qquad\leq \frac L T\norm{S}_{L^4(\mu)}^2\norm{\Pi \Gamma_{\mathcal R}} \int_{0}^{T} \ee^{-\lambda t}\, \d t 
         \xrightarrow[T \to \infty]{} 0.
    \end{align*}
    The result then follows from the equality $\expect_{\mu} \bigl[ \Gamma_{\mathcal R} \bigr] = 2\myip{R, -\mathcal L^{-1} R}$,
    see for example~\cite[Theorem 6.12]{pavliotis2008multiscale}.
    The latter expression can be justified by rewriting the carré du champ as as $\Gamma_{\mathcal R} = \mathcal L \mathcal R^2 - 2 \mathcal R \mathcal L \mathcal R$ and integrating against~$\mu$.
\end{proof}
\section{Variance reduction using importance sampling}
\label{sec:variance-reduction}
In the previous section, we showed that the variance of the time-truncated estimator of the Green--Kubo formula grows linearly in time,
with a constant prefactor including the term $\myip{R,-\mathcal L^{-1} R}$.
Thus, a natural idea to achieve variance reduction is to perturb the reference dynamics in such a manner that~$\myip{R,-\mathcal L_{\alpha}^{-1} R}$ is smaller than $\myip{R,-\mathcal L^{-1} R}$,
where $\mathcal L_{\alpha}$ is the generator of the modified dynamics.
In order to define an estimator based on the modified dynamics that is asymptotically consistent as the integration time~$T$ tends to~$+\infty$,
weights must be included in accordance with Girsanov's theorem.
We first present in~\cref{sec:new_approx} an estimator of~$\rho$ in~\eqref{girsanov-application} based on the modified dynamics.
Then, in~\cref{sec:asymptotic}, we study the dependence of the variance of the estimator in terms of the biasing parameter $\alpha$,
and provide an approximation of the value of~$\alpha$ which minimizes the variance.
Finally, we investigate the dependence of the optimal biasing parameter with respect to the integration time~$T$.
\subsection{Approximation based on Girsanov formula}
\label{sec:new_approx}
In this section we study an approximation for the time-truncated Green--Kubo formula obtained using Girsanov's theorem.
Let us place ourselves in a probability space $(\Omega, \mathcal F, \proba)$,
and let $(q_t^{\alpha})$ denote the solution to the overdamped Langevin dynamics
with biasing force $\alpha \sigma u\colon \torus^d \to \real^d$ for~$\alpha \in \real$:
\begin{equation}
    \label{eq:biased_dynamics}
    \d q_t^\alpha = b(q_t^\alpha) \, \d t + \alpha \sigma(q_t^\alpha) u(q_t^\alpha) \,\d t + \sigma(q_t^\alpha)\, \d W^{\proba}_t, \qquad q_0^\alpha = q_0,
\end{equation}
where $(W^{\proba}_{t\geq0})$ is a standard Brownian motion with respect to~$\proba$
and $u\colon \real^d \to \real^m$ is a vector field.
We emphasize that this equation is considered with the same initial condition as~\eqref{eq:dyn_brow}. For notational convenience, 
the upperscrit is omitted for the reference (unbiased) dynamics, \textit{i.e.}~$q_t = q_t^0$.
Let also $\mathcal E(X^{\alpha})_t$ denote the Doléans-Dade exponential of the process $X^{\alpha}$,
which is given by
\begin{align}
    \label{eq:expression_expo}
    \mathcal E(X^{\alpha})_t
    = \exp \left( X^{\alpha}_t - \frac{1}{2} \langle X^{\alpha} \rangle_t \right),
     \qquad X^{\alpha}_t = - \alpha\int_0^t u(q_s^\alpha) \cdot \d W_s^\proba.
\end{align}
We show in~\cref{proposition:first_second_moment_girsanov} that,
as a consequence of Girsanov's theorem,
it holds
\begin{align}
    \label{girsanov-application}
    \expect \left[ \int_0^T R(q_t)S(q_0) \, \d t \right]
    = \expect \left[\mathcal{E}(X^{\alpha})_T \int_0^TR(q_t^\alpha)S(q_0^\alpha) \, \d t \right].
\end{align}
Thus we define the following estimator:
\begin{align}
    \label{eq:many_replica}
    \widehat{A}_{T,J}^{\alpha} = \frac{1}{J}\sum_{j=1}^J \mathcal E(X^{\alpha,j})_T \int_0^T R(q_t^{\alpha,j})S(q_0^{\alpha,j}) \, \d t,
\end{align}
where $\bigl(q_t^{\alpha,j}\bigr)$,
for $j \in \{1, \dotsc, J\}$,
are $J$ independent realization of the dynamics~\eqref{eq:biased_dynamics} and
\[
    X^{\alpha,j}_t = -\alpha\int_0^t u(q_s^{\alpha,j}) \cdot \d W^{j,\proba}_s .
\]
It is clear that
\(
    \expect \left[ \widehat{A}_{T,J}^{\alpha} \right] = \expect \left[ \widehat{A}_{T}^{\alpha} \right],
\)
and
\(
    \var \left[ \widehat{A}_{T,J}^{\alpha} \right] = \frac{1}{J} \var \left[ \widehat{A}_{T}^{\alpha} \right],
\)
where
\begin{align}
    \label{eq:good_estimator}
    \widehat{A}_T^{\alpha}  =  \mathcal{E}(X^\alpha)_T \int_0^T R(q_t^\alpha)S(q_0^\alpha)\,\d t.
\end{align}
In other words, the many-replica estimator~\eqref{eq:many_replica} has the same bias as the one-replica estimator~\eqref{eq:good_estimator},
and a smaller variance by a factor~$\frac{1}{J}$.
Therefore, it is sufficient to study the estimator~\eqref{eq:good_estimator}.
The main goal is now to understand whether this estimator can be employed for variance reduction purposes.
To this end,
we first show the following result on the raw first and second moments of~$\widehat A_T^{\alpha}$.
\begin{proposition}
    \label{proposition:first_second_moment_girsanov}
    For all $T \ge 0$ and $\alpha \in \mathbb{R}$,
    \begin{equation}
        \label{eq:bias_girsanov}
        \expect \left[ \widehat{A}_T^{\alpha} \right]
        = \expect \left[ \widehat{\rho}_T \right].
    \end{equation}
    Furthermore,
    \begin{align}
        \label{eq:variance_girsanov}
        \expect \left[ \left\lvert \widehat{A}_T^{\alpha} \right\rvert^2 \right]
        = \expect \left[\left(\int_0^T R(q_t)S(q_0)\,\d t \right)^2 E_T(\alpha)\right],
    \end{align}
    with
    \[
        E_T(\alpha)
        =  \exp\left(-\alpha\int_0^T u(q_t) \cdot \d W^{\proba}_t
        + \frac{\alpha^2}{2}\int_0^T \bigl\lvert u(q_t) \bigr\rvert^2 \, \d t\right).
    \]
\end{proposition}
\begin{proof}
    Equality~\eqref{eq:bias_girsanov} is a direct application of Girsanov's theorem. 
    We nonetheless make precise how to obtain this result as this will be useful to establish~\eqref{eq:variance_girsanov}.

    Since $b$ is bounded,
    the so-called Novikov condition
    is fulfilled (see~\cite[Exercise~4.4]{MR2001996}):
    \[
        \expect\left[\exp\left(\frac12\int_0^T |X^{\alpha}_t|^2 \d t\right)\right]< \infty.
    \]
    Consequently, the stochastic process $t \mapsto \mathcal E(X^{\alpha})_t$ is a martingale.
    Thus, it follows from Girsanov's theorem,
    see e.g.~\cite[Theorem~8.1]{MR3097957} and \cite[Theorems 8.6.4 and~8.6.8]{MR2001996},
    that the probability measure~$\proba^\alpha$ defined by its Radon--Nykodym derivative
    \begin{equation}
        \label{eq:girsanov0}
        \frac{\d \proba^\alpha}{\d \proba} (\omega)
        = \mathcal E(X^{\alpha})_T
    \end{equation}
    is equivalent to $\proba$.
    Girsanov's theorem also gives that the process
    \begin{equation}
        \label{eq:girsanov1}
        W^{\proba^\alpha}_t  := \alpha \int_{0}^{t} u(q^{\alpha}_t) \, \d t  + W^{\proba}_t,
    \end{equation}
    is a standard Brownian motion with respect to $\proba^\alpha$,
    and that, furthermore,
    the process $(q_{t}^{\alpha})_{t\geq0}$ is a strong solution to
    \begin{equation}
        \label{eq:girsanov2}
        \d q_t^\alpha = b(q_t^\alpha) + \sigma \, \d W^{\proba^\alpha}_t, \qquad q_0^\alpha = q_0.
    \end{equation}
    It immediately follows that~\eqref{eq:bias_girsanov} holds,
    as
    \begin{align*}
        \expect_{\proba} \left[\mathcal E(X^{\alpha})_T \int_{0}^{T} R(q^{\alpha}_t)S(q^{\alpha}_0)\,\d t \right]
        &\stackrel{\eqref{eq:girsanov0}}{=} \expect_{\proba^\alpha} \left[ \int_{0}^{T} R(q^{\alpha}_t)S(q^{\alpha}_0)\,\d t \right]\\
        &\stackrel{\eqref{eq:girsanov2}}{=}
        \expect_{\proba} \left[ \int_{0}^{T} R(q_t)S(q_0)\,\d t \right].
    \end{align*}
    In the last equation,
    we used that the $\proba$-law of the process~$(q_t)_{t\geq0}$
    is equal to the $\proba^{\alpha}$-law of the process~$(q^{\alpha}_t)_{t\geq0}$.
    This is clear in view of~\eqref{eq:girsanov2} and the fact that the $\proba^{\alpha}$-law of~$q^{\alpha}_0$ is
    equal to the $\proba$-law of $q_0$,
    since $q^{\alpha}_0 = q_0$ and $q_0$ is independent of~$\mathcal E(X^{\alpha})_t$. \\
    To prove~\eqref{eq:variance_girsanov},
    we use~\eqref{eq:girsanov1} to rewrite
    \begin{align}
        \notag
        \mathcal E(X^{\alpha})_T
        &= \exp\left(- \alpha \int_0^T u(q^{\alpha}_t) \, \d W^{\proba}_t
        - \frac{\alpha^2}{2} \int_0^T \bigl\lvert u(q_t^\alpha) \bigr\rvert^2 \, \d t \right) \\
        \label{eq:rewritten_weight}
        &= \exp\left(- \alpha \int_0^T u(q^{\alpha}_t) \, \d W^{\proba^\alpha}_t
        + \frac{\alpha^2}{2} \int_0^T \bigl\lvert u(q_t^\alpha) \bigr\rvert^2 \, \d t \right).
    \end{align}
    Therefore, by~\eqref{eq:girsanov0} and~\eqref{eq:rewritten_weight}
    \begin{align*}
        &\expect_{\proba} \left[\left(\mathcal{E}(X^{\alpha})_T\int_0^T R(q_t^\alpha)S(q_0^\alpha) \, \d t \, \right)^2\right]
        = \expect_{\proba^\alpha} \left[\mathcal{E}(X^{\alpha})_T \left(\int_0^T R(q_t^\alpha)S(q_0^\alpha) \, \d t \, \right)^2 \right] \\
        &= \expect_{\proba^\alpha}
        \left[
            \exp
            \left(
                - \alpha \int_0^T u(q^\alpha_t) \, \d W^{\proba^\alpha}_t
                + \frac{\alpha^2}{2} \int_0^T \bigl\lvert u(q^\alpha_t) \bigr\rvert^2 \, \d t
            \right)
            \left(\int_0^T R(q^{\alpha}_t)S(q^{\alpha}_0) \, \d t \, \right)^2
        \right].
    \end{align*}
    Since the $\proba^\alpha$ law of $(W^{\proba^\alpha}_t, q^{\alpha}_t)_{t \in [0, T]}$ coincides with
    the $\proba$ law of $(W^{\proba}_t, q_t)_{t \in [0, T]}$,
    we finally obtain~\eqref{eq:variance_girsanov}.
\end{proof}
In view of~\eqref{eq:bias_girsanov},
it holds that
\begin{align}
    \label{eq:variance_estimator}
    \var \left[ \widehat{A}_{T}^{\alpha} \right]
    = \expect\left[ \left| \widehat{A}_{T}^{\alpha} \right|^2\right] - \expect\left[ \widehat{A}_{T}^{\alpha}\right]^2
    = \expect\left[ \left| \widehat{A}_{T}^{\alpha} \right|^2\right] - \expect\left[ \widehat \rho_{T}\right]^2.
\end{align}
The second term of this expression does not depend on $\alpha$ and converges to $\rho^2$ in the limit~$T \to \infty$,
which motivates studying the scaling with respect to~$\alpha$ of the first term.
\Cref{proposition:first_second_moment_girsanov}, specifically equation~\eqref{eq:variance_girsanov},
is very useful for this purpose,
since it gives an expression in terms of the solution to the unbiased dynamics~\eqref{eq:dyn_brow}.
Let us introduce $F_T\colon \real \to \real$ given by
\begin{align}
    F_T(\alpha)= \expect \left[ \left\lvert \widehat{A}_T^{\alpha} \right\rvert^2 \right] = \expect \left[\Phi_T(q) \, \exp\left( \alpha X_T + \frac{\alpha^2}{2} \langle X \rangle_T \right)\right],
\end{align}
where
\begin{align}
    \label{eq:x_t}
    X_T = -\int_0^T u(q_t) \cdot \d W_t,
    \qquad \langle X \rangle_T = \int_0^T \bigl\lvert u(q_t) \bigr\rvert^2 \, \d t,
\end{align}
and
\begin{align}
    \label{eq:definition_phi}
    \Phi_T(q) = \left(\int_{0}^T R(q_t) S(q_0) \, \d t\right)^2.
\end{align}
For a fixed time $T$, the map $F_T$ is a real function of a single variable.
The next natural step is to show the existence and uniqueness of an optimal $\alpha \in \real$ that minimizes the variance. This is the purpose of the next proposition.
Let us first introduce the following additional assumptions.
\begin{assumption}
    \label{assump:fc_nonzero}
    The observables~$R$ and~$S$ are continuous and non-zero, and so is the perturbation $u$.
\end{assumption}
\begin{remark}
    It is in fact sufficient to assume that the observables are only~$\mathcal{C}^1(\torus^d)$ in~\cref{assump:fc_nonzero}. 
    For simplicity's sake we retain the smoothness assumption.
\end{remark}
\begin{assumption}
    \label{assump:mes_inv_ci}
    The initial condition~$q_0 = q^{\alpha}_0$ of every dynamics considered in the sequel is sampled from the invariant probability measure~$\mu$.
    In particular, this implies that for every time $t \in \real_+$,
    the distribution of~$q_t$ is~$\mu$.
\end{assumption}
\begin{proposition}
    \label{prop:differentiation}
    Under~\cref{assump:fc_nonzero,assump:mes_inv_ci}, for all $T>0$, the function $\alpha \mapsto F_T(\alpha)$ is smooth,
    with first, second and third derivatives given by
    \begin{subequations}
        \begin{align}
            \label{eq:first_derivative}
            F_T'(\alpha) &= \expect \left[ \Phi_T(q) \bigl(X_T + \alpha \langle X \rangle_T \bigr) \exp\left( \alpha X_T + \frac{\alpha^2}{2} \langle X \rangle_T \right) \right], \\
            \notag
            F_T''(\alpha) &= \expect \left[\Phi_T(q) \langle X \rangle_T \exp\left( \alpha X_T + \frac{\alpha^2}{2} \langle X \rangle_T \right) \right]  \\
            \label{eq:second_derivative}
                          &\qquad + \expect \left[\Phi_T(q)\left(\alpha\langle X \rangle_T + X_T \right)^2 \exp\left( \alpha X_T + \frac{\alpha^2}{2} \langle X \rangle_T \right) \right], \\
                          \notag
            F^{(3)}_T(\alpha) &= \expect_\mu\!\left[\Phi_T(q)\,\bigl(X_T + \alpha\langle X \rangle_T\bigr)^3 \exp\left( \alpha X_T + \frac{\alpha^2}{2} \langle X \rangle_T \right)\right] \\
            \label{eq:third_derivative}
            &\quad+ 3\,\expect_\mu\!\left[\Phi_T(q)\,\langle X \rangle_T\left(X_T + \alpha\langle X \rangle_T\right) \exp\left( \alpha X_T + \frac{\alpha^2}{2} \langle X \rangle_T \right)\right].
        \end{align}
    \end{subequations}
    Furthermore, the function $F_T$ is strictly convex and coercive.
    Therefore, there exists a unique global minimizer $\alpha_T^*$ characterized by
    \(
        F_T'(\alpha_T^*) = 0.
    \)
\end{proposition}
\begin{proof}
    We decompose the proof into 3 steps corresponding to the three main properties of $F_T$.

    \paragraph{Differentiability}
    We show that for all $T > 0$, the function $F_T$ is differentiable.
    To this end, we use a dominated convergence argument to differentiate under the expectation,
    see~\cite[Theorem~2.27]{MR1681462}.
    Introduce
    \[
        g(\alpha) = \Phi_T(q) \exp\left( \alpha X_T + \frac{\alpha^2}{2} \langle X \rangle_T \right).
    \]
    For a fixed realization $\omega \in \Omega$,
    it is clear that the function $\alpha \mapsto g(\alpha)$ is differentiable with derivative
    \begin{equation}
        \label{eq:pointwise_derivative}
        \frac{\partial g}{\partial \alpha} (\alpha)
        = \Phi_T(q) \bigl(X_T + \alpha \langle X \rangle_T\bigr) \exp\left( \alpha X_T + \frac{\alpha^2}{2} \langle X \rangle_T \right).
    \end{equation}
    Fix $m \in (0, \infty)$.
    In order to show~\eqref{eq:first_derivative},
    it suffices to show that there exist a random variable~$Y$ and a constant~$C \geq 0$
    independent of~$\alpha$ such that,
    for all $\alpha \in [-m, m]$,
    it holds
    \begin{align}
        \label{eq:deriv_sous_int}
        \expect\bigl\lvert g(\alpha) \bigr\rvert \leq C < \infty,
        \qquad
        \left\lvert \frac{\partial g}{\partial \alpha} (\alpha) \right\rvert \leq Y,
        \qquad
        \expect [Y] < \infty.
    \end{align}
    For the first inequality,
    we note that
    \begin{align}
        \notag
        \expect \bigl\lvert g(\alpha) \bigr\rvert &\leq
        \bigl(\|S\|_{L^\infty}\|R\|_{L^\infty}T \bigr)^2 \expect\!\left[\exp\left( \alpha X_T + \frac{\alpha^2}{2} \langle X \rangle_T \right) \right] \\
        \notag
        &\leq \bigl(\|S\|_{L^\infty}\|R\|_{L^\infty}T \bigr)^2 \expect\!\left[\exp\left( \alpha X_T - \frac{\alpha^2}{2} \langle X \rangle_T \right) \exp\left( \alpha^2\langle X \rangle_T \right)\right] \\
        \label{eq:intermediaire_jpp}
        &\leq \bigl(\|S\|_{L^\infty}\|R\|_{L^\infty}T \bigr)^2 \exp\left( m^2 \|u\|_{L^\infty}^2T\right) \expect\!\left[\exp\left( \alpha X_T - \frac{\alpha^2}{2} \langle X \rangle_T \right) \right] =: C.
    \end{align}
    The term on the right-hand side is the expectation
    of the Doléans-Dade exponential at $T$ of $X$,
    which is equal to 1 so that~$C$ in~\eqref{eq:deriv_sous_int} is given by~\eqref{eq:intermediaire_jpp}.
    In view of~\eqref{eq:pointwise_derivative},
    we have, for all~$\alpha \in [-m, m]$,
    \begin{align*}
        \left\lvert \frac{\partial g}{\partial \alpha} (\alpha) \right\rvert
        &\leq
            \mathbf 1_{\{X_T \geq 0\}} \Phi_T(q) \Bigl( X_T + m \langle X \rangle_T \Bigr)
            \exp\left( m X_T + \frac{m^2}{2} \langle X \rangle_T \right) \\
        &\qquad +
            \mathbf 1_{\{X_T < 0\}} \Phi_T(q) \Bigl( \lvert X_T \rvert  + m \langle X \rangle_T \Bigr)
            \exp\left( \frac{m^2}{2} \langle X \rangle_T \right)
        =: Y_1 + Y_2.
    \end{align*}
    It remains to show that $\expect[Y] < \infty$ with $Y = Y_1 + Y_2$.
    We only prove that $\expect[Y_1] < \infty$,
    since the result for $Y_2$ is obtained by a similar reasoning.
    By the Cauchy--Schwarz inequality,
    
    \begin{align*}
        \expect[Y_1]
        \leq
        \expect \left[ \Phi_T(q)^{4} \right]^{\frac{1}{4}}
        \expect \Bigl[ \bigl( X_T + m \langle X \rangle_T \bigr)^4 \Bigr] ^{\frac{1}{4}}
        \expect \Bigl[ \exp\left( 2m  \lvert X_T \rvert + m^2 \langle X \rangle_T \right) \Bigr]^{\frac{1}{2}}.
    \end{align*}
    The first two factors are finite,
    because $R$, $S$ and $u$ are bounded,
    and moments of the It\^o integral~$X_T$ are finite by the Burkholder--Davis--Gundy inequality~\cite{MR2380366} since $u$ is bounded. 
    Indeed, 
    we have, 
    for any~$p\in [1,+\infty)$ and~$g \in\mathcal{C}^\infty(\torus^d)$,
    \begin{align}
        \label{eq:bdg}
        \expect\!\left[\left|\int_{0}^{T}g(q_t)\,\d W_t\right|^p\right] \leq \cbdg{p} \expect\!\left[\left(\int_{0}^{T}\left|g(q_t)\right|^2 \,\d t\right)^{p/2}\right].
    \end{align}
    The last factor is also bounded since
    \begin{align*}
        \expect \Bigl[ \exp\left( 2m X_T + m^2 \langle X \rangle_T \right) \Bigr]
        &= \expect \Bigl[ \exp \left( 2 m^2 \langle X \rangle_T \right) \exp\left( 2m X_T - 2 m^2 \langle X \rangle_T \right) \Bigr] \\
        &\leq \expect \Bigl[ \exp \left( 2 C m^2 \right) \exp\left( 2m X_T - 2 m^2 \langle X \rangle_T \right) \Bigr],
    \end{align*}
    with~$C := T \norm{u}_{L^{\infty}}^2$.
    The expectation of the second exponential on the right-hand side is equal to 1,
    as this term is an exponential martingale by~\cite[Section 3.7]{MR3288096}.
    A similar reasoning,
    which we do not make precise,
    applies for higher order derivatives,
    giving the smoothness of~$F_T$ and the formulas~\eqref{eq:second_derivative} and~\eqref{eq:third_derivative}.
\paragraph{Strict convexity}
    Both terms in the definition~\eqref{eq:second_derivative} of $F_T''(\alpha)$ are nonnegative which already shows that $F_T$ is convex.
    To show that the convexity is strict,
    suppose there is~$\alpha$ such that~$F_T''(\alpha) = 0$.
    Then
    \[
        \expect \left[\Phi_T(q) \langle X \rangle_T \exp\left( \alpha X_T + \frac{\alpha^2}{2} \langle X \rangle_T \right) \right] = 0,
    \]
    and so the argument of the expectation, being non-negative,
    must be zero almost surely.
    The exponential is almost surely non-zero,
    so
    \[
        \langle X \rangle_T \Phi_T(q) = \langle X \rangle_T \left( \int_{0}^T R(q_t) S(q_0) \, \d t \right)^2 \stackrel{\rm a.s.}{=} 0,
    \]
    and so
    \[
        \langle X \rangle_T S(q_0) \int_{0}^T R(q_t) \, \d t
        \stackrel{\rm a.s.}{=} 0.
    \]
    Given~\cref{assump:fc_nonzero},
    it is possible to consider $(q_S, q_R, q_U) \in \torus^d \times \torus^d \times \torus^d$ and $\varepsilon > 0$ such that
    \begin{align*}
        &\forall q \in B(q_S, \varepsilon), \qquad \bigl\lvert S(q) \bigr\rvert > 0, \\
        &\forall q \in B(q_R, \varepsilon), \qquad \bigl\lvert R(q) \bigr\rvert \geq \frac{1}{2}\norm{R}_{L^{\infty}(\mu)}, \\
        & \forall q \in B(q_U, \varepsilon), \qquad \bigl\lvert u(q) \bigr\rvert > 0.
    \end{align*}
    In particular, since $R$ is continuous,
    it does not change sign in $B(q_R, \varepsilon)$.
    Given~\cref{assump:mes_inv_ci} and using a controllability argument of the same type as in~\cite[Lemma~3.4]{MR1931266},
    which hinges on the fact that the probability that a Brownian motion stays confined in a narrow tube around any continuous function starting at $0$ is strictly positive,
    it is possible to show that~$\proba \left[ A \right] > 0$,
    with
    \[
        A := \left\{q_0 \in B(q_S, \varepsilon) \,\, \text{and} \,\, q_{T/8} \in B(q_U, \varepsilon) \,\, \text{and} \,\, q_t \in B(q_R, \varepsilon) \,\, \text{for all} \,\, \frac{T}{4} \leq t \leq T\right\}
    \]
    This leads to a contradiction,
    because for any $\omega$ in this event of positive probability,
    it holds by construction that $\langle X \rangle_T > 0$ and $S(q_0) \neq 0$ and
    \[
        \left\lvert \int_{0}^T R(q_t) \, \d t \right\rvert
        \geq \left\lvert \int_{\frac{T}{4}}^T R(q_t) \, \d t \right\rvert
        - \left\lvert \int_{0}^{\frac{T}{4}} R(q_t) \, \d t \right\rvert
        \geq \frac{3T}{8} \norm{R}_{L^{\infty}(\mu)} - \frac{T}{4} \norm{R}_{L^{\infty}(\mu)} > 0.
    \]
    \paragraph{Coercivity}
    Fix $T>0$.
    We showed previously that the event
    \[
        A \subset B := \Bigl\{\omega \in \Omega : \Phi_T(q) > 0 \,\, \text{and} \,\, \langle X \rangle_T > 0 \Bigr\}.
    \]
    has positive probability.
    Let
    \[
        S_T(\alpha) = \mathbf 1_{B} \Phi_T(q) \exp\left( \alpha X_T + \frac{\alpha^2}{2} \langle X_t \rangle_T \right),
    \]
    and note that
    \begin{equation}
        \label{eq:limits}
        \lim_{|\alpha| \to \infty} S_T(\alpha)
        =
        \begin{cases}
            \infty \quad & \text{if $\omega \in B$}, \\
            0 \quad & \text{if $\omega \notin B$}. \\
        \end{cases}
    \end{equation}
    Consider a sequence $(\alpha_n)_{n \in \nat}$ with~$\alpha_n \rightarrow +\infty$ as~$n \rightarrow +\infty$.
    Then, 
    by definition of $F_T(\alpha)$,
    Fatou's lemma and~\eqref{eq:limits},
    it holds that
    \[
        \liminf_{n \to \infty} F_T(\alpha_n)
        \geq
        \liminf_{n \to \infty} \expect \bigl[ S_T(\alpha_n) \bigr]
        \geq \expect \bigl[ \liminf_{n\to \infty} S_T(\alpha_n) \bigr]
        = \proba[A] \times \infty = \infty,
    \]
    which gives the coercivity. 
    \paragraph{Conclusion of the proof} The function $F_T$ is strictly convex and coercive,
    which implies the existence and uniqueness of $\alpha^{*}_T$ which minimizes $F_T$,
    and concludes the proof.
\end{proof}
\subsection{Estimator of the optimal value}
\label{sec:asymptotic}
\cref{prop:differentiation} establishes the existence and uniqueness of a value~$\alpha^*_T$
for which that the variance of the Girsanov estimator~\eqref{eq:many_replica} is minimized.
In this section,
which is the main contribution of this work,
we study the asymptotic behavior with respect to the time $T$ of this optimal value.
Because the function $F_T$ is smooth
and $\alpha_T^*$ is expected to vanish in the limit $T \rightarrow +\infty$,
we anticipate that a Taylor expansion of the form
\begin{align}
    \label{eq:taylor_naive}
    F_T'(\alpha_T^*) \simeq F_T'(0) + \alpha_T^* F_T''(0)
\end{align}
should be accurate when $T \gg 1$.
Such an expansion is convenient as it allows us to define an approximation~$\widehat \alpha^*_T$ of~$\alpha^*_T$
as the root of the right-hand side of~\eqref{eq:taylor_naive},
namely
\begin{align}
    \label{estimator-alpha}
    \widehat{\alpha}_T^* = -\frac{F_T'(0)}{F_T''(0)}\cdot
\end{align}
This approximation is well-defined given that $F''_T(0) \neq 0$ by~\cref{proposition:first_second_moment_girsanov}.
The value~\eqref{estimator-alpha} can be calculated based on the unbiased dynamics only.
By using the explicit expressions of~$F'_T(0)$ and~$F''_T(0)$ given in~\eqref{eq:first_derivative} and~\eqref{eq:second_derivative},
we obtain
\begin{align}
    \label{eq:explicit_approx_alpha}
    \widehat{\alpha}_T^* = \frac{\displaystyle\expect\left[\Phi_T(q)\int_0^T u(q_t)\cdot\d W_t\right]}{\displaystyle\expect\!\left[ \Phi_T(q)\left(\int_0^T |u(q_t)|^2\,\d t +\left(\int_0^T u(q_t)\cdot\d W_t\right)^2\right)\right]},
\end{align}
with~$\Phi_T$ defined in~\eqref{eq:definition_phi}. 
The remainder of this section is organized as follows.
In~\cref{sec:asymptotic_behavior},
we study the asymptotic behavior with respect to~$T$ of $F_T'(0)$ and $F_T''(0)$ (see~\cref{lemmaN,lemmaD}),
as well as that of the third derivative~$F_T^{(3)}(\xi)$ for small~$\xi$ (see~\cref{lemme:deriv3}).
These estimates enable to prove, in~\cref{sec:limit-communation},
quantitative estimates on $\widehat \alpha^*_T$ and then~$\alpha^*_T$.
\subsubsection{Asymptotic estimations}
\label{sec:asymptotic_behavior}
By a Taylor expansion, for all $a \in \real$ there exists $\xi_{a,T} \in \left[-\frac{|a|}{T}, \frac{|a|}{T}\right]$ such that
\begin{align}
    \label{eq:taylor_expansion}
    F_T\left(\frac{a}{T}\right) - F_T(0) = \frac{a}{T}F'_T(0) + \frac{a^2}{2 T^2}F''_T(0) + \frac{a^3}{6 T^3}F^{(3)}_T(\xi_{a,T}).
\end{align}
The three following lemmas (\cref{lemmaN,lemmaD,lemme:deriv3}) give the scaling with respect to~$T$ of each of the three derivatives on the right-hand side of~\eqref{eq:taylor_expansion}.
\begin{lemma}
    \label{lemmaN}
    Under~\cref{assumption:inv_decay,assump:fc_nonzero,assump:mes_inv_ci} the following limit holds:
    \begin{align}
        \lim_{T \to \infty} \frac{F'_T(0)}{T}
        = D_1 ,
          \label{eq:estimate_num}
    \end{align}
    where
    \begin{align*}
        D_1 &= -2\langle S^2, \mathcal{R}\rangle\langle u, \sigma^\top\nabla\mathcal{R}\rangle \\
        &\quad+\|S\|^2\Bigl(2\left\langle \sigma^\top\nabla\mathcal{L}^{-1}\Pi\left(u^\top\sigma^\top\nabla \mathcal{R}\right), \sigma^\top\nabla \mathcal{R}\right\rangle + \left\langle \sigma^\top\nabla\mathcal{L}^{-1}\Pi\left(|\sigma^\top\nabla \mathcal{R}|^2\right), u\right\rangle\Bigr),
    \end{align*}
    with $\mathcal{R}$ the unique solution in~$L^2_0(\mu)$ of the Poisson equation $-\mathcal{L}\mathcal{R} = R$.
\end{lemma}
\begin{proof}
    Note that the Poisson equation is well-posed because~$R \in L^2_0(\mu)$.
    Moreover by elliptic regularity $\mathcal{R}\in\mathcal{C}^\infty(\torus^d)$, 
    so that~$\norm{\mathcal{R}}_{L^\infty}<+\infty$ in particular.
    Recall from~\cref{prop:differentiation} that
    \begin{equation}
        \label{eq:first_equation_num}
        -F'_T(0) = \expect\!\left[S(q_0)^2 \int_0^T u(q_t) \cdot \d W_t \left(\int_0^T R(q_t) \, \d t \right)^2\right].
    \end{equation}
    We use It\^o's formula to write
    \[
        \int_0^T R(q_t) \,\d t = -\int_0^T \mathcal{L}\mathcal{R}(q_t) \,\d t = -\Bigl(\mathcal{R}(q_T) - \mathcal{R}(q_0)\Bigr)
        +  \int_0^T \nabla\mathcal{R}(q_t)^\top \sigma(q_t) \, \d W_t.
    \]
    By substituting in~\eqref{eq:first_equation_num}, we obtain
    \begin{align*}
        -F'_T(0) &= \expect\!\left[S(q_0)^2 \int_0^T u(q_t) \cdot \d W_t \Bigl(\mathcal{R}(q_T) - \mathcal{R}(q_0)\Bigr)^2\right] \\
        & \quad - 2\expect\!\left[S(q_0)^2 \int_0^T u(q_t) \cdot \d W_t \,\Bigl(\mathcal{R}(q_T) - \mathcal{R}(q_0)\Bigr) \int_0^T \nabla\mathcal{R}(q_t)^\top \sigma(q_t) \, \d W_t\right] \\
        & \quad + \expect\!\left[S(q_0)^2 \int_0^T u(q_t) \cdot \d W_t \,\left(\int_0^T \nabla\mathcal{R}(q_t)^\top \sigma(q_t)\,\d W_t\right)^2\right] \\
        &=: A_1(T) + A_2(T) + A_3(T).
    \end{align*}
    \paragraph{Bounding the first term}
    Since $u$, $S$ and $\mathcal R$ are bounded uniformly on the torus,
    we obtain by the Cauchy--Schwarz inequality and It\^o's isometry that
    \begin{align*}
        \frac1T \left\lvert A_1(T)\right\rvert & \le \frac1T \sqrt{\expect\!\left[S(q_0)^4  \Bigl(\mathcal{R}(q_T) - \mathcal{R}(q_0)\Bigr)^4\right]} \sqrt{\expect\!\left[\int_0^T | u(q_t) |^2 \d t \right]} \\
        & \le \frac4T \|S\|_{L^\infty}^2  \|\mathcal{R}\|_{L^\infty}^2 \|u\|_{L^\infty} \sqrt{T} \xrightarrow[T \to \infty]{} 0.
    \end{align*}

    \paragraph{Bounding the second term}
    The second term can itself be divided into two contributions:
    \begin{align*}
        A_2(T) &=  -2\expect\!\left[S(q_0)^2 \mathcal{R}(q_T) \int_0^T u(q_t) \cdot \d W_t  \int_0^T \nabla\mathcal{R}(q_t)^\top \sigma(q_t) \, \d W_t\right] \\
        & \quad + 2\expect\!\left[S(q_0)^2 \mathcal{R}(q_0) \int_0^T u(q_t) \cdot \d W_t  \int_0^T \nabla\mathcal{R}(q_t)^\top \sigma(q_t) \, \d W_t \right]\\
        & =: -2 B_1(T) + 2 B_2(T).
    \end{align*}
    For~$B_1(T)$, let us introduce~$\displaystyle I_T = \int_0^T u(q_t) \cdot \d W_t$
    and~$\displaystyle J_T = \int_0^T \nabla\mathcal{R}(q_t)^\top \sigma(q_t) \, \d W_t$.
    We have that
    $$
    I_T J_T =
    I_{T - \sqrt{T}} J_{T - \sqrt{T}} + \left( I_T J_T - I_{T-\sqrt{T}} J_{T-\sqrt{T}} \right).
    $$
    Let $(\mathcal F_t^q)_{t\geq 0}$ denote the natural filtration generated by $(q_t)_{t\geq0}$.
    By conditioning on~$\mathcal F_{T-\sqrt{T}}$,
    the sigma algebra which contains all information up to time $T - \sqrt{T}$,
    then using the tower property of conditional expectation
    and the definition of the semigroup~\eqref{eq:def_semigroup},
    we obtain
    \begin{align*}
        \frac1T\expect\!\left[S(q_0)^2 \mathcal{R}(q_T) I_{T - \sqrt{T}} J_{T - \sqrt{T}}\right]
        & = \frac1T\expect \left[ \expect \left[ S(q_0)^2  I_{T - \sqrt{T}} J_{T - \sqrt{T}} \mathcal R(q_{T}) \, \middle| \, \mathcal F_{T - \sqrt{T}}\right] \right] \\
        & = \frac1T\expect \left[ S(q_0)^2  I_{T - \sqrt{T}} J_{T - \sqrt{T}} \expect \left[  \mathcal R(q_{T}) \, \middle| \, \mathcal F_{T - \sqrt{T}}\right] \right] \\
        & = \frac1T\expect \left[ S(q_0)^2  I_{T - \sqrt{T}} J_{T - \sqrt{T}} \left( \ee^{\sqrt{T}\mathcal L}\mathcal{R} \right) (q_{T-\sqrt{T}}) \right].
    \end{align*}
    Noting that the random variable~$q_{T - \sqrt{T}}$ is distributed according to~$\mu$ by stationarity~(\cref{assump:mes_inv_ci}),
    and using~\eqref{eq:bdg},
    we obtain
    \begin{align*}
        &\left\lvert \frac1T\expect\!\left[S(q_0)^2 \mathcal{R}(q_T) I_{T - \sqrt{T}} J_{T - \sqrt{T}}\right] \right\rvert \\
         &\qquad\le  \frac{\|S\|^2_{L^\infty}}{T} \left(\expect\!\left[|I_{T - \sqrt{T}}|^4\right] \right)^{\frac14}\left(\expect\!\left[|J_{T - \sqrt{T}}|^4\right]\right)^{\frac14} \left\|\ee^{\sqrt{T}\mathcal L}\mathcal{R}\right\| \\
         &\qquad\le \cdecay \sqrt{\cbdg{4}}\|S\|^2_{L^\infty} \|u\|_{L^\infty} \|\sigma^\top\nabla \mathcal R\|_{L^\infty} \|\mathcal R\|\frac{T-\sqrt{T}}{T}\ee^{-\lambda\sqrt{T}},
    \end{align*}
    and the right-hand side tends to 0 as~$T \to \infty$.
    Here $\cbdg{4}$ is the constant from the BDG inequality~\eqref{eq:bdg},
    and we used \cref{assumption:inv_decay} on the decay of the semigroup.
    Next, writing
    \[
        I_T J_T - I_{T-\sqrt{T}} J_{T-\sqrt{T}} = I_T(J_{T} - J_{T-\sqrt{T}}) + J_{T-\sqrt{T}}(I_{T} - I_{T-\sqrt{T}}),
    \]
    we have, 
    by the Cauchy--Schwarz inequality and It\^o's isometry,
    \begin{align*}
        \left\lvert \frac1T\expect\Bigl[S(q_0)^2 \mathcal{R}(q_T) I_T(J_{T} - J_{T-\sqrt{T}})\Bigr] \right\rvert
        &\leq \frac1T\expect\Bigl[S(q_0)^2 \lvert \mathcal{R}(q_T) \rvert \, \lvert I_T \rvert \, \lvert J_{T} - J_{T-\sqrt{T}} \rvert\Bigr] \\
        & \le \frac{\|S\|^2_{L^\infty}  \| \mathcal R \|_{L^{\infty}}}{T} \sqrt{\expect\!\left[|I_T|^2\right]} \sqrt{\expect\!\left[|J_{T} - J_{T-\sqrt{T}}|^2\right]}\\
        & \le \frac{1}{T} \|S\|^2_{L^\infty} \|\mathcal R\|^2_{L^\infty} \|u\|_{L^\infty} \sqrt{T} \| \sigma^\top\nabla \mathcal R \|_{L^\infty} T^{1/4} \\
        & \xrightarrow[T \to \infty]{} 0.
    \end{align*}
    The same computations hold for the term with $J_{T-\sqrt{T}}(I_{T} - I_{T-\sqrt{T}})$,
    and so
    \begin{align*}
        \lim_{T \to \infty} \frac{B_1(T)}{T} = 0.
    \end{align*}
    For $B_2(T)$, we use Itô's isometry and~\cref{assumption:inv_decay} to obtain
    \begin{align*}
        \frac{B_2(T)}{T} &= \frac1 T \expect\!\left[S(q_0)^2 \mathcal{R}(q_0) \int_0^T \nabla\mathcal{R}(q_t)^\top  \sigma(q_t)  u(q_t) \, \d t \right] \\
                         &= \langle S^2, \mathcal{R}\rangle \langle u, \sigma^{\top} \nabla\mathcal{R}\rangle
                         + \frac{1}{T} \int_{0}^{T} \expect\!\left[ S(q_0)^2 \mathcal{R}(q_0) \Pi \left( \nabla\mathcal{R}^\top  \sigma  u \right)(q_t) \right] \, \d t \\
                         &= \langle S^2, \mathcal{R}\rangle \langle u, \sigma^\top \nabla\mathcal{R}\rangle
                         + \frac{1}{T} \int_{0}^{T} \Bigl\langle S^2 \mathcal{R}, \e^{t \mathcal L} \Pi \left( \nabla\mathcal{R}^\top  \sigma  u \right) \Bigr\rangle \, \d t \\
                         &\xrightarrow[T \to \infty]{} \langle S^2, \mathcal{R}\rangle \langle u, \sigma^\top \nabla\mathcal{R}\rangle.
    \end{align*}
    Therefore we have shown that
    \begin{align*}
        \lim_{T \to \infty} \frac{A_2(T)}{T}
            = 2\langle S^2, \mathcal{R}\rangle \langle u, \sigma^\top \nabla\mathcal{R}\rangle.
    \end{align*}
    \paragraph{Bounding the third term}
    As before, let $\expect^{q_0}$ denote the expectation with respect to all realizations of the Brownian motion $W_t$ in \eqref{eq:dyn_brow},
    for a dynamics starting at a deterministic initial condition~$q_0$.
    By conditioning on~$q_0$ we first note that
    \[
        \frac{A_3(T)}{T} = \expect\!\left[S(q_0)^2 \frac{\mathcal H_T(q_0)}{T} \right],
    \]
    with
    \[
    \mathcal H_T(q_0) := \expect^{q_0}\!\left[ \int_0^T u(q_t) \cdot \d W_t \,\left(\int_0^T \nabla\mathcal{R}(q_t)^\top \sigma(q_t) \, \d W_t\right)^2 \right].
    \]
    Introduce the solutions of the following two Poisson equations:
    \begin{align*}
        -\mathcal{L}\mathcal{F} = \mathfrak{f}, \qquad -\mathcal{L}\mathcal{G} = \mathfrak{g},
    \end{align*}
    with $\mathfrak{f} = \Pi\left( \nabla \mathcal{R}^\top \sigma u\right)$ and $\mathfrak{g} = \Pi \left|\sigma^\top \nabla \mathcal R \right|^2$.
    Hereafter we first show that
    \begin{equation}
        \label{eq:intermediate_third_term}
        \forall q_0 \in \torus^d, \qquad
        \frac{\mathcal H_T(q_0)}{T} \xrightarrow[T \to \infty]{} 2\left\langle\sigma^\top \nabla \mathcal{F}, \sigma^\top  \nabla \mathcal{R}\right\rangle + \left\langle\sigma^\top\nabla \mathcal{G}, u\right\rangle.
    \end{equation}
    To this end, using~\cref{lemme:egalite-moment-3} with $f = u$ and~$g = h = \sigma^\top \nabla\mathcal{R}$,
    we have that
    \begin{align}
        \notag
        \mathcal H_T(q_0)
        &= 2\expect^{q_0}\!\left[ \int_0^T \nabla \mathcal{R}(q_t)^\top \sigma(q_t) u(q_t) \,\d t\,\int_0^T \nabla\mathcal{R}(q_t)^\top \sigma(q_t) \, \d W_t\right] \\
        \label{eq:to_cauchy_schwarz}
        &\hspace{1cm}+ \expect^{q_0}\left[ \int_0^T \Bigl\lvert \sigma(q_t)^\top \nabla \mathcal{R}(q_t) \Bigr\rvert^2\,\d t \,\int_0^T u(q_t)\cdot\d W_t\right].
    \end{align}
    Because $\displaystyle\expect^{q_0}\left[\int_0^T \nabla\mathcal{R}(q_t)^\top \sigma(q_t) \, \d W_t\right] = \expect^{q_0}\left[ \int_0^T u(q_t)\cdot\d W_t\right] = 0$,
    it holds that
    \begin{align*}
        \mathcal H_T(q_0)
        &= 2\expect^{q_0}\left[ \int_0^T\mathfrak{f}(q_t)\,\d t\,\int_0^T \nabla\mathcal{R}(q_t)^\top \sigma(q_t) \, \d W_t\right] \\
        &\quad + \expect^{q_0}\left[ \int_0^T \mathfrak{g}(q_t)\,\d t\,\int_0^T u(q_t)\cdot\d W_t\right].
    \end{align*}
    We can again apply It\^o's formula to obtain
    \begin{align*}
        &\expect^{q_0}\left[ \int_0^T \mathfrak{f}(q_t)\,\d t\,\int_0^T \nabla\mathcal{R}(q_t)^\top \sigma(q_t) \,\d W_t\right] = \\
        &\qquad\qquad\qquad\expect^{q_0}\left[ \Bigl(\mathcal{F}(q_0) - \mathcal{F}(q_T)\Bigr)\int_0^T \nabla\mathcal{R}(q_t)^\top \sigma(q_t) \, \d W_t\right] \\
        &\qquad\qquad\qquad + \expect^{q_0}\left[\int_0^T \nabla\mathcal{F}(q_t)^\top \sigma(q_t) \, \d W_t\,\int_0^T\nabla\mathcal{R}(q_t)^\top \sigma(q_t) \, \d W_t\right].
    \end{align*}
    The first term on right-hand side is of order $\sqrt{T}$ at most since the functions~$\mathcal{F}$ and~$\sigma^\top\nabla\mathcal{R}$ are bounded on~$\torus^d$.
    Indeed, 
    by a Cauchy--Schwarz inequality and It\^o's isometry,
    \begin{align*}
        &\left\lvert \expect^{q_0}\left[ \Bigl(\mathcal{F}(q_0) - \mathcal{F}(q_T)\Bigr)\int_0^T \nabla\mathcal{R}(q_t)^\top \sigma(q_t) \, \d W_t\right]\right\rvert \\
        &\qquad\qquad\leq 2\|\mathcal{F}\|_{L^\infty} \left( \expect^{q_0}\left[ \left\lvert \int_0^T \sigma(q_t)\nabla\mathcal{R}(q_t)\cdot\d W_t\right\rvert^2 \right] \right)^{\frac{1}{2}}\\
        & \qquad\qquad = 2\|\mathcal{F}\|_{L^\infty} \left( \expect^{q_0}\left[  \int_0^T \left\lvert \nabla\mathcal{R}(q_t)^\top \sigma(q_t) \right\rvert^2\d t\right] \right)^{\frac{1}{2}}\\
        &\qquad\qquad \leq 2\|\mathcal{F}\|_{L^\infty} \|\sigma^\top \nabla\mathcal{R}\|_{L^\infty} \sqrt{T}.
    \end{align*}
    For the second term, by It\^o's isometry and ergodicity,
    \begin{align*}
        \lim_{T \to \infty}
        \frac1T \expect^{q_0}\left[\int_0^T \nabla\mathcal{F}(q_t)^\top \sigma(q_t) \, \d W_t\,\int_0^T\nabla\mathcal{R}(q_t)^\top \sigma(q_t) \, \d W_t\right]
        = \left\langle\sigma^\top \nabla \mathcal{F}, \sigma^\top \nabla \mathcal{R}\right\rangle.
    \end{align*}
    With similar computations, we obtain that
    \begin{align*}
        \lim_{T \to \infty} \frac1T \expect^{q_0}\left[ \int_0^T |\sigma(q_t)\nabla \mathcal{R}(q_t)|^2\,\d t\,\int_0^T u(q_t)\cdot\d W_t\right] = \left\langle\sigma^\top \nabla \mathcal{G}, u\right\rangle.
    \end{align*}
    We have thus proved~\eqref{eq:intermediate_third_term}.
    Using dominated convergence (the function~$\mathcal{H}_T/T$ being uniformly bounded, 
    which can be shown by computations similar to the ones used above),
    we finally have
    \begin{align*}
        \lim_{T \to \infty} \frac{A_3(T)}{T} = \|S\|^2\bigl(2\left\langle\sigma^\top \nabla \mathcal{F}, \sigma^\top  \nabla R\right\rangle + \left\langle\sigma^\top\nabla \mathcal{G}, u\right\rangle\bigr).
    \end{align*}
    \paragraph{Conclusion of the proof}
    By combining all the previous results, we can finally conclude that
    \begin{align*}
        \lim_{T \to \infty} \frac{F'_T(0)}{T} & = -2\langle S^2, \mathcal{R}\rangle \langle u, \sigma^\top\nabla\mathcal{R}\rangle -\|S\|^2\bigl(2\left\langle\sigma^\top\nabla \mathcal{F}, \sigma \nabla \mathcal{R}\right\rangle + \left\langle\sigma^\top\nabla \mathcal{G}, u\right\rangle\bigr),
    \end{align*}
    which gives the claimed expression for~$D_1$.
\end{proof}
The next lemma provides an asymptotic equivalent for the second derivative of~$F_T$ at~$0$.
The proof is slightly different from the one for the first derivative,
as it requires more sophisticated limit theorems.
\begin{lemma}
    \label{lemmaD}
    Under~\cref{assumption:inv_decay,assump:fc_nonzero,assump:mes_inv_ci} the following limit holds:
        \begin{align}
            \label{eq:equivalent_second_derivative}
            \lim_{T \to \infty}\frac{F''_T(0)}{T^2} = D_2,
         \end{align}
        where $D_2 = 2 \norm{S}^2 \left(\norm{\sigma^\top\nabla \mathcal{R}}^2 \norm{u}^2 + \myip{\sigma^\top\nabla \mathcal{R}, u}^2  \right)$,
        with $\mathcal{R}$ the unique solution in~$L^2_0(\mu)$ of the Poisson equation $- \mathcal L \mathcal{R} = R$.
        \newline
    \end{lemma}
    \begin{proof}
        Recall that \eqref{eq:second_derivative} implies that
        \[
            F''_T(0) = \expect\!\left[ \left(\int_0^T |u(q_t)|^2 \, \d t +\left(\int_0^T u(q_t)\cdot\d W_t\right)^2\right)  \left(\int_{0}^{T} S(q_0) R(q_t) \, \d t\right)^2 \right].
        \]
        Introduce~$I_T = \displaystyle\int_{0}^{T} R(q_t) \, \d t$,
        so that~$F_T''(0)$ can be rewritten as
        \begin{align*}
            F''_T(0) =\,
            & T \norm{u}^2 \expect\Bigl[ S(q_0)^2 \lvert I_T \rvert^2 \Bigr]
            + \expect\Bigl[ S(q_0)^2 \lvert I_T \rvert^2 X_T^2 \Bigr] \\
            &+ \expect\Bigl[ S(q_0)^2 \lvert I_T \rvert^2 \int_0^T \Pi|u(q_t)|^2 \, \d t  \Bigr]
            =:\, A_1(T) + A_2(T) + A_3(T),
        \end{align*}
        where~$X_T$ is defined in~\cref{eq:x_t}.
        \paragraph{Bounding the first term}
        By It\^o's formula,
        it holds that
        \begin{equation}
            \label{eq:It\^o_formula}
            I_T = \mathcal{R}(q_0) - \mathcal{R}(q_T) + M_T,\qquad M_T = \int_{0}^T \nabla \mathcal{R}(q_t)^\top \sigma(q_t) \, \d W_t.
        \end{equation}
        Therefore,
        \begin{align}
            \notag
            \frac{A_1(T)}{T^2}
            = &\norm{u}^2 \expect\!\left[ S(q_0)^2 \left\lvert \frac{\mathcal{R}(q_0) - \mathcal{R}(q_T)}{\sqrt{T}} \right\rvert^2 \right]
            +  \norm{u}^2 \expect\!\left[  S(q_0)^2  \left(\frac{M_T}{\sqrt{T}}\right)^2 \right] \\
            \label{eq:first_term_three_terms}
              &+ 2\norm{u}^2 \expect\!\left[ S(q_0)^2  \left(\mathcal{R}(q_0) - \mathcal{R}(q_T)\right) \frac{M_T}{\sqrt{T}} \right].
        \end{align}
        The first term on the right-hand side of~\eqref{eq:first_term_three_terms} converges to 0 as~$T \to \infty$ because~$S$ and~$\mathcal{R}$ are bounded.
        For the second term,
        we use It\^o's isometry and the limit~$\ee^{t\mathcal{L}}f \rightarrow \int_{\torus^d} f \d \mu$ in~$L^2(\mu)$ for all $f \in L^2(\mu)$ to obtain
        \begin{align*}
            \norm{u}^2 \expect\!\left[ S(q_0)^2  \left(\frac{M_T}{\sqrt{T}}\right)^2 \right]
            &= \norm{u}^2 \expect\!\left[S(q_0)^2 \frac{1}{T} \int_{0}^T \lvert \sigma(q_t)^\top \nabla \mathcal{R}(q_t) \rvert^2 \, \d t \right] \\
            &\xrightarrow[T \to \infty]{}\norm{u}^2 \norm{S}^2 \norm{\sigma^\top \nabla \mathcal{R}}^2.
        \end{align*}
        Finally, by a Cauchy--Schwarz inequality,
        and the two previous convergence results, the third term in~\eqref{eq:first_term_three_terms} tends to 0.
        Therefore,
        \[
            \lim_{T \to \infty} \frac{A_1(T)}{T^2}
            = \norm{u}^2 \norm{S}^2 \norm{\sigma^\top \nabla \mathcal{R}}^2.
        \]
        \paragraph{Bounding the second term}
        For the second term, we note that
        \begin{align}
            \notag
            \frac{A_2(T)}{T^2}
            &= \expect\!\left[ S(q_0)^2 \left\lvert \frac{I_T}{\sqrt{T}} \right\rvert^2 \, \left\lvert \frac{Y_T}{\sqrt{T}} \right\rvert^2 \right]
            = \expect \Bigl[ S(q_0)^2 \mathcal H_T(q_0) \Bigr],
        \end{align}
        with~$\mathcal H_T(q_0) = \displaystyle\expect^{q_0}\!\left[ \left\lvert \frac{I_T}{\sqrt{T}} \right\rvert^2 \, \left\lvert \frac{Y_T}{\sqrt{T}} \right\rvert^2 \right]$. 
        By the central limit theorem for additive functionals of Markov processes
        (see~\cref{corollary:central_limit})
        it holds for all $q_0 \in \torus^d$ that
        \[
            \begin{pmatrix}
                Y_1^T \\
                Y_2^T
            \end{pmatrix}
            :=
            \begin{pmatrix}
            I_T / \sqrt{T} \\
            Y_T / \sqrt{T}
            \end{pmatrix}
            \xrightarrow[T \to \infty]{\rm Law}
            Y^{\infty} \sim
            \mathcal N (0, \Sigma), \qquad
                \Sigma :=
                \begin{pmatrix}
                    \|\sigma^\top \nabla \mathcal{R}\|^2 & \left\langle \sigma^\top \nabla \mathcal{R} , u\right\rangle \\
                \left\langle \sigma^\top \nabla \mathcal{R} , u\right\rangle &  \|u\|^2
                \end{pmatrix}.
        \]
        Using the same reasoning as in~\cite[Appendix C]{pavliotis2008multiscale},
        together with Isserlis' theorem,
        we have that~$\mathcal H_T(q_0)$ converges for every fixed~$q_0$ to a constant in the limit~$T \to \infty$:
        \begin{align*}
            \lim_{T \to \infty} \mathcal H_T(q_0)
            &= \expect \Bigl[\lvert Y^{\infty}_1 \rvert^2 \lvert Y^{\infty}_2 \rvert^2 \Bigr]
            =  \expect\Bigl[\lvert Y^{\infty}_1 \rvert^2 \Bigr] \expect \Bigl[\lvert Y^{\infty}_2 \rvert^2 \Bigr]
            + 2 \Bigl\lvert \expect \bigl[Y^{\infty}_1  Y^{\infty}_2\bigr] \Bigr\rvert^2 \\
            &=   \norm{\sigma \nabla \mathcal{R}}^2 \norm{u}^2
            + 2 \myip{\sigma \nabla \mathcal{R}, u}^2.
        \end{align*}
        Applying the dominated convergence theorem,
        we deduce that
        \[
            \lim_{T \to \infty} \frac{A_2(T)}{T^2} = \expect_{\mu}\!\left[ S(q_0)^2 \lim_{T \to \infty} \mathcal H_T(q_0) \right]
            = \norm{S}^2 \left(\norm{\sigma^\top\nabla \mathcal{R}}^2 \norm{u}^2 + 2\myip{\sigma^\top\nabla \mathcal{R}, u}^2 \right).
        \]
        The dominated convergence theorem can indeed be applied since~$\mathcal{H}_T(q_0)$ is uniformly bounded. 
        Indeed,
        using~\eqref{eq:It\^o_formula}, then the elementary inequality $(a + b)^4 \leq 8a^4 + 8b^4$
        and finally the Burkholder--Davis--Gundy inequality,
        we have
        \begin{align*}
            \expect^{q_0}\!\left[ \left\lvert \frac{I_T}{\sqrt{T}} \right\rvert^2 \, \left\lvert \frac{X_T}{\sqrt{T}} \right\rvert^2\right]
            &\leq \frac{1}{2T^2} \left( \expect^{q_0}\!\left[ \left\lvert I_T \right\rvert^4 + \left\lvert X_T \right\rvert^4 \right] \right) \\
            &\leq \frac{1}{2T^2} \left( \expect^{q_0}\!\left[ 8 \left\lvert M_T \right\rvert^4 + \lvert X_T \rvert^4 \right] + 128 \|\mathcal{R}\|_{L^\infty}^4 \right) \\
            &\leq \frac{1}{2T^2} \left(\cbdg{4} \expect^{q_0}\!\left[ 8\langle M \rangle_T^2 + \langle X \rangle_T^2  \right] + 128 \|\mathcal{R}\|_{L^\infty}^4 \right) \\
            &\leq \frac{1}{2T^2}\left(8\cbdg{4} \|\sigma^\top \nabla\mathcal{R}\|^4_{L^\infty} T^2 + \cbdg{4} \|u\|^4_{L^\infty} T^2 + 128 \|\mathcal{R}\|_{L^\infty}^4 \right) \\
            &\leq \frac{1}{2}\left(8 \cbdg{4} \|\sigma^\top \nabla\mathcal{R}\|^4_{L^\infty} +  \cbdg{4} \|u\|^4_{L^\infty} + \frac{128 \|\mathcal{R}\|_{L^\infty}^4}{T^2}\right).
        \end{align*}
        \paragraph{Bounding the third term}
        Let us introduce the solution $\psi \in L^1_0(\mu)$ to the Poisson equation~$- \mathcal L \psi = \Pi\lvert u \rvert^2$,
        where~$\psi \in\mathcal{C}^\infty(\torus^d)$ by elliptic regularity.
        By It\^o's formula, it holds that
        \[
            \int_{0}^T \Pi\lvert  u(q_t) \rvert^2 \, \d t
            = \psi(q_0) - \psi(q_T) + \int_{0}^T \nabla \psi(q_t)^\top \sigma(q_t) \, \d W_t,
        \]
        and so
        \begin{align*}
            A_3(T)
            &= \expect\!\left[ S(q_0)^2 \lvert I_T \rvert^2  \int_{0}^{T} \nabla \psi(q_t)^\top \sigma(q_t) \, \d W_t \right]
            + \expect\Bigl[ S(q_0)^2 \lvert I_T \rvert^2  \bigl(\psi(q_0) - \psi(q_T)\bigr) \Bigr] \\
            &= C_1(T) + C_2(T).
        \end{align*}
        In view of~\cref{proposition:asymptotic_variance_gk} and the boundedness  of~$\psi$,
        it holds for the second term that $C_2(T) / T^2 \to 0$ in the limit~$T \to \infty$.
        For the first term, using again~\eqref{eq:It\^o_formula}, Young's inequality,
        and the Burkholder--Davis--Gundy inequality,
        we deduce that~$C_1(T)$ is of order~$T^{3/2}$ as~$T \rightarrow + \infty$,
        and therefore
        \[
            \lim_{T \to \infty} \frac{A_3(T)}{T^2} = 0,
        \]
        which concludes the proof.
    \end{proof}
The final lemma gives a bound on the third derivative on the right-hand side of~\eqref{eq:taylor_expansion}.
The goal is to show that this term is negligible when divided by~$T^3$.
\begin{lemma}
    \label{lemme:deriv3}
    Under~\cref{assump:fc_nonzero}, 
    for any $a >0$,
    \begin{align*}
        \lim_{T \to \infty} \frac{1}{T^3}\left(\underset{\xi \in [-\frac{a}{T},\frac{a}{T}]}\sup \left\lvert F^{(3)}_T\left(\xi\right)\right\rvert\right) = 0.
    \end{align*}
\end{lemma}
\begin{proof}
    Recall the expression~\eqref{eq:third_derivative} for~$F^{(3)}_T(\alpha)$:
    \begin{align*}
        F^{(3)}_T(\alpha) &= \expect\!\left[\Phi_T(q)\,\bigl(X_T + \alpha\langle X \rangle_T\bigr)^3E_T(\alpha)\right]+ 3\,\expect\!\left[\Phi_T(q)\,\langle X \rangle_T\left(X_T +\alpha\langle X \rangle_T\right) E_T(\alpha)\right],
    \end{align*}
    with the notation
    \begin{align*}
        E_T(\alpha) = \exp\left( \alpha X_T + \frac{\alpha^2}{2} \langle X \rangle_T \right).
    \end{align*}
    Note first that
    \begin{align*}
        \left\lvert F^{(3)}_T(\alpha)\right\rvert & \leq  \Psi_T(\alpha) + 3\Theta_T(\alpha),
    \end{align*}
    where
    \begin{align*}
        &\Psi_T(\alpha) := \expect\!\left[\Phi_T(q)\,\bigl(|X_T| + \alpha\langle X \rangle_T\bigr)^3 E_T(\alpha)\right],\\ 
        &\Theta_T(\alpha) :=\expect\!\Bigl[\Phi_T(q)\,\langle X \rangle_T\Bigl(|X_T| + \alpha\langle X \rangle_T\Bigr) E_T(\alpha)\Bigr].
    \end{align*}
    By Hölder's inequality,
    \begin{align*}
        \Psi_T(\alpha) \le \left(\expect\!\left[\Phi_T(q)^3\,\right] \expect\!\left[\left(|X_T| + \alpha\langle X \rangle_T\right)^9\right] \expect\!\left[E_T(\alpha)^3\right]\right)^{\frac{1}{3}}.
    \end{align*}
    We next bound each of the three terms in the product on the right-hand side.
    \paragraph{Bounding the first term of the product}
    We recall the definition of~$\Phi_T(q)$ in~\eqref{eq:definition_phi} and introduce the solution $\mathcal R \in L^2_0(\mu)$ of the Poisson equation $-\mathcal{L}\mathcal{R} = R$.
    By using It\^o's formula we obtain
    \begin{align*}
        \expect\!\left[\Phi_T(q)^3\right]
        & = \expect\!\left[\left(\int_0^T R(q_t)S(q_0)\, \d t\right) ^6\right]
        \le \|S\|^6_{L^\infty}\expect\!\left[\left(\int_0^T \mathcal{L}\mathcal{R}(q_t)\, \d t\right) ^6\right]
     \\ & \le \|S\|^6_{L^\infty}\expect\!\left[\left(\mathcal{R}(q_T)- \mathcal{R}(q_0) - \int_0^T \nabla\mathcal{R}(q_t)\cdot \d W_t\right) ^6\right].
    \end{align*}
    We now apply the elementary inequality $|x+y|^p \le 2^{p-1}\Bigl(|x|^p + |y|^p\Bigr)$ for $p \ge 1$ to write
    \begin{align*}
        \expect\!\left[\Phi_T(q)^3\right]
        & \le 2^{5}\|S\|^6_{L^\infty}\expect\!\left[\Bigl(\mathcal{R}(q_T)- \mathcal{R}(q_0)\Bigr)^6 + \left(\int_0^T \nabla\mathcal{R}(q_t)\cdot \d W_t\right) ^6\right].
    \end{align*}
    Since $\mathcal R$ is bounded on the torus,
    and using the Burkholder--David--Gundy inequality~\eqref{eq:bdg},
    it follows that
    \begin{align*}
        \expect\!\left[\Phi_T(q)^3\,\right] & \le 2^{5}\|S\|^2_{L^\infty}\left(2^{6}\|\mathcal{R}\|^6_{L^\infty}+\cbdg{6}\expect_\mu\!\left[ \left(\int_0^T |\nabla\mathcal{R}(q_t)|^2\, \d t\right) ^3\right]\right) \\
                                            &\leq 2^{5}\|S\|^2_{L^\infty}\left(2^{6}\|\mathcal{R}\|^6_{L^\infty}+\cbdg{6} \|\nabla\mathcal{R}\|^6_{L^\infty}T^3 \right).
    \end{align*}
    \paragraph{Bounding the second term of the product} With manipulations similar to the ones used for the first term, we obtain that
    \begin{align*}
        \expect\!\left[\left(|X_T| + \alpha\langle X \rangle_T\right)^9\right] & \leq 2^8\left(\expect\!\left[|X_T|^9\right] +\alpha^9\expect\!\left[\langle X \rangle_T^9\right]\right) \\
        &\leq 2^8\left(\cbdg{9}\expect_\mu\!\left[\langle X \rangle_T^{9/2}\right] +\alpha^9\expect\!\left[ \langle X \rangle_T^9\right]\right) \\
        & \leq 2^8 \left(\cbdg{9}+\alpha^9\|u\|_{L^\infty}^9 T^{9/2}\right)\|u\|_{L^\infty}^9 T^{9/2}.
    \end{align*}
   Finally,
   \begin{align*}
    \expect\!\left[\left(|X_T| + \alpha\langle X \rangle_T\right)^9\right]^{1/3} \leq 2^{8/3} \left(\cbdg{9}+\alpha^9\|u\|_{L^\infty}^9 T^{9/2}\right)^{1/3}\|u\|_{L^\infty}^3 T^{3/2}.
   \end{align*}
    \paragraph{Bounding the third term of the product}
    We can write
    \begin{align*}
        \expect\!\left[E_T(\alpha)^3\right] & = \expect\!\left[\exp\left(3\alpha X_T + \frac32\alpha^2\langle X\rangle_T \right)\right] \\
        &= \expect\!\left[\exp\left(3\alpha X_T - \frac{9}{2}\alpha^2\langle X\rangle_T \right)\exp\left(6\alpha^2\langle X\rangle_T\right)\right] \\
        & \leq \exp(6\alpha^2\|u\|^2_{L^\infty}T)\,\expect\!\left[\exp\left(3\alpha X_T - \frac{9}{2}\alpha^2\langle X\rangle_T \right)\right]
    \end{align*}
    The Novikov condition is satisfied,
    so the Doléans-Dade exponential has an expectation of $1$. 
    Therefore
    \begin{align*}
        \expect\!\left[E_T(3\alpha)\right]^{\frac{1}{3}} & \le \exp\Bigl(2\alpha^2\|u\|^2_{L^\infty}T\Bigr).
    \end{align*}
    Finally, by combining the three inequalities, we have
    \begin{align}
        \label{eq:ineq_deriv3}
        \notag
        \Psi_T(\alpha) &\le 2^{13/3}\|S\|^{2/3}_{L^\infty}\left(\frac{2^6\|\mathcal{R}\|^6_{L^\infty}}{T^3}+\cbdg{6} \|\nabla\mathcal{R}\|^6_{L^\infty} \right)^{1/3} \\
        & \quad \times \left(\cbdg{9}+\alpha^9\|u\|_{L^\infty}^9 T^{9/2}\right)^{1/3}\|u\|_{L^\infty}^3 T^{5/2}\exp(2\alpha^2\|X\|^2_{L^\infty}T).
    \end{align}
    Then by taking~$\alpha = \frac{a}{T}$ and applying this inequality for $\xi \in [0,\frac{a}{T}]$, by dividing by $T^3$ and bounding $\xi$ by~$\frac{a}{T}$ in \eqref{eq:ineq_deriv3}, we conclude that
    \begin{align*}
        \underset{\xi \in [-\frac{a}{T},\frac{a}{T}]}\sup\, \frac{1}{T^3}\Psi_T(\xi) &\leq \frac{2^{13/3}}{\sqrt{T}}\|S\|^{2/3}_{L^\infty}\left(\frac{2^6\|\mathcal{R}\|^6_{L^\infty}}{T^3}+\cbdg{6} \|\nabla\mathcal{R}\|^6_{L^\infty} \right)^{1/3} \\
        & \quad \times \left(\cbdg{9}+a^9\|u\|_{L^\infty}^9 T^{-9/2}\right)^{1/3}\|u\|_{L^\infty}^3 \exp(2a^2\|X\|^2_{L^\infty}T^{-1}) \\
        & \quad \xrightarrow[T \to \infty]{} 0.
    \end{align*}
    Similar estimates can be used to prove the convergence of~$\Theta_T$ and conclude to the claimed convergence result.
\end{proof}
\subsubsection{Asymptotic behavior of the optimal value}
\label{sec:limit-communation}
We now come back to~\eqref{eq:taylor_expansion}.
By using \cref{lemmaN,lemmaD,lemme:deriv3},
we obtain that for all fixed~$a \in \real$
\begin{align}
    \label{eq:lim_F_centered}
    \lim_{T \to \infty} F_T\left(\frac{a}{T}\right) - F_T(0) = D_1 a + \frac{D_2}{2} a^2,
\end{align}
where $D_1$ and $D_2$ are the constants introduced in from \cref{lemmaN,lemmaD}.
The next step is to understand how the optimal value $\alpha^*_T$ behaves.
Equation \eqref{eq:lim_F_centered} suggests that $\alpha^*_T$ scales as $\frac{1}{T}$,
which we justify rigorously based on the following result.
\begin{proposition}
    \label{prop:limits_exchange}
    Under~\cref{assumption:inv_decay,assump:fc_nonzero,assump:mes_inv_ci},
    \begin{equation}
        \label{eq:convergence_argmin}
        \lim_{T \to \infty} \argmin_{a \in \real}\left\{F_T\left(\frac{a}{T}\right) - F_T(0)\right\} = \argmin_{a \in \real} \left\{D_1 a + \frac{D_2}{2} a^2\right\}.
    \end{equation}
    Furthermore, 
    as a consequence of~\eqref{eq:convergence_argmin}
    \begin{equation}
        \label{eq:conv_min}
        \lim_{T \to \infty} \left(\min_{\alpha \in \real} F_T\left(\alpha\right) - F_T(0)\right)  = - \frac{D_1^2}{2D_2}.
    \end{equation}
\end{proposition}
\begin{proof}
    Fix a sequence $(t_n)_{n \in \nat}$ such that $t_n \to \infty$,
    and consider the sequence of functions $\left(f_n\right)_{n \in \mathbb N}$ defined as
    \begin{align}
        \label{eq:suite_fn}
        f_n(a) = F_{t_n}\left(\frac{a}{t_n}\right) - F_{t_n}(0).
    \end{align}
    This is a sequence of strictly convex coercive functions converging pointwise to $f_{\infty}(a) = D_1 a + D_2 a^2 / 2$ on $\real$ by~\eqref{eq:lim_F_centered}.
    Therefore, by~\cite[Theorem 10.8]{rockafellar2015convex},
    this sequence converges uniformly to~$f_{\infty}$ on every compact subset of~$\real$. 
    It is then straightforward to show 
    \iffalse
    either that this sequence of functions is equicoercive and then apply~\cite[Corollary~7.24]{dal1993introduction} or
    \fi 
    that~$\left(\argmin_{a \in \real} f_n(a)\right)_{n\in\mathbb{N}}$ is contained in a compact set, which allows to prove~\eqref{eq:convergence_argmin}. 
    We give a proof of this result in~\cref{lem:min_comp}.
    As a consequence,~\eqref{eq:conv_min} directly follows from~\eqref{eq:convergence_argmin}.
\end{proof}
As a direct corollary of \cref{prop:limits_exchange},
we have that
\begin{align}
    \label{eq:equi_alpha_opt}
    \alpha^*_T := \argmin_{\alpha \in \real}\,F_T\left(\alpha\right) \underset{T \to \infty}\sim -\frac{D_1}{D_2 T}.
\end{align}
This is the main result of this paper.
The latter equation gives us the scaling with respect to~$T$ of the parameter~$\alpha$ that minimizes the variance of the Girsanov estimator~\eqref{eq:many_replica}.
Equation~\eqref{eq:conv_min} in~\cref{prop:limits_exchange},
together with~\cref{proposition:asymptotic_variance_gk},
also imply that
\begin{align}
    \label{eq:gain_var_rel}
    \frac{F_T(0) - F_T(\alpha^*_T)}{F_T(0)}
    \underset{T \to \infty} \sim \frac{D_1^2}{4 D_2 \|S\|^2\left\langle R, -\mathcal{L}^{-1}R \right\rangle} \frac{1}{T}.
\end{align}
This last result shows that the variance can only be reduced by a factor that scales asymptotically as~$T^{-1}$.
\section{Numerical results}
\label{sec:numerical-results}
We illustrate in this section the theoretical predictions obtained in~\cref{sec:replin-and-GK,sec:variance-reduction}.
In~\cref{sec:discrete_girsanov},
we present the numerical method used to perform the numerical computations,
and motivate the discretization formula for the Girsanov weights.
Then,
in~\cref{sec:numerical_experiments},
we consider toy problems in dimensions 1 and 2.
Throughout this section, we consider for simplicity only the case where~$\sigma$ is a constant positive scalar.
\subsection{Numerical method}
\label{sec:discrete_girsanov}
At the continuous-time level,
Girsanov's theorem ensures that
the following equality holds for any functional~$\mathcal F \colon C\bigl([0, T], \real^d\bigr) \to \real$ on path space:
\begin{align}
    \label{eq:girsanov-application}
    \expect \left[ \mathcal F\bigl(q_{[0, T]}\bigr) \right]
    = \expect \Bigl[\mathcal{E}(X^{\alpha})_T \mathcal F\Bigl(q^{\alpha}_{[0, T]}\Bigr) \Bigr].
\end{align}
In the context of this paper,
the functional of interest is the truncated Green--Kubo integral:
\[
    \mathcal F\bigl(q_{[0, T]}\bigr) = \int_{0}^{T} R(q_t) S(q_0) \, \d t.
\]
We consider Euler--Maruyama discretizations of
the reference dynamics $(q_t)_{t\geq0}$ given in~\eqref{eq:dyn_brow} and the biased dynamics~$(q^{\alpha}_t)_{t\geq0}$ given in~\eqref{eq:biased_dynamics}. 
We denote by~$\qapprox{n}$ and~$\qapproxa{n}$ the iterates of the numerical schemes \textit{i.e.} $\qapprox{n}$ and~$\qapproxa{n}$ are approximations of~$q_{n\Delta t}$ and~$q_{n\Delta t}^\alpha$:
\begin{subequations}
\begin{align}
    \label{eq:disc_ref}
    \qapprox{n+1} &= \qapprox{n} + b(\qapprox{n})\Delta t + \sigma\Delta W_n, \\
    \label{eq:disc_biased}
    \qapproxa{n+1}
            &= \qapproxa{n} + \Bigl(b(\qapproxa{n}) +\alpha \sigma u(\qapproxa{n})\Bigr) \Delta t + \sigma\Delta W_n 
            =: \qapproxa{n} + b^\alpha(\qapproxa{n})\Delta t + \sigma\Delta W_n,
\end{align}
\end{subequations}
where~$(\Delta W_n)_{n\in\mathbb{N}}$ are i.i.d Gaussian random variables with mean~$0$ and covariance~$\Delta t\, \mathrm{I}_d$.
In order to implement the weighted Green--Kubo estimator~\eqref{eq:many_replica},
it is also essential to discretize the Green--Kubo functional~$\mathcal F$ and the Girsanov weight.
A discretization of the functional, denoted by~$\widehat {\mathcal F} \colon \real^{d(N+1)} \to \real$,
is obtained for instance by an approximation of the time integral using the left point rule:
\[
    \widehat {\mathcal F} \bigl(\qapprox{0}, \dots, \qapprox{N}\bigr) =
    \sum_{n=0}^{N} R(\qapprox{n}) S(\qapprox{0}) \Delta t.
\]
Ideally, the discretization of the Girsanov weight should be consistent with the discretization of the dynamics in~\eqref{eq:disc_ref} and \eqref{eq:disc_biased},
in the sense that given finite discretized trajectories~$(\qapprox{0}, \dots, \qapprox{N})$ and~$(\qapproxa{0}, \dots, \qapproxa{N})$,
it should hold
\begin{equation}
    \label{eq:discrete_girsanov_formula}
    \expect \left[ \widehat {\mathcal F}\bigl(\qapprox{0}, \dots, \qapprox{N}\bigr) \right]
    =
    \expect \Bigl[ M^{\alpha}_N\Bigl(\qapproxa{0}, \dots, \qapproxa{N}\Bigr)
    \widehat{\mathcal F}\Bigl(\qapproxa{0}, \dots, \qapproxa{N}\Bigr) \Bigr],
\end{equation}
where $M^{\alpha}_N$ is the discretized Girsanov weight.
In the next paragraph, we motivate the discretization of this weight,
relying on the calculations provided in the appendix of~\cite{DHK}.
For the Euler--Maruyama discretization considered in~\eqref{eq:disc_ref} and~\eqref{eq:disc_biased},
the calculation of the likelihood ratio~$M^{\alpha}_N$ is relatively simple.
Let $\mu_{\Delta t,N}$ and $\mu^\alpha_{\Delta t,N}$ denote the Lebesgue densities % of the laws
of the random vectors~$(\qapprox{0}, \dots, \qapprox{N})$ and~$(\qapproxa{0}, \dots, \qapproxa{N})$,
respectively.
In order for~\eqref{eq:discrete_girsanov_formula} to hold,
the weight $M_N^\alpha$ must coincide with the quotient $\mu_{\Delta t,N} / \mu^\alpha_{\Delta t,N}$.
Now the transition kernel describing~\cref{eq:disc_ref} has a density
\begin{align}
    \label{eq:proba_chgmt_etat}
    \mathscr{P}(\qapprox{n+1} | \qapprox{n}) = \frac{1}{\sqrt{2\pi \Delta t \sigma^2}}\exp\!\left(-\frac{1}{2\Delta t\sigma^2}\left\lvert \qapprox{n+1}- \qapprox{n} - b(\qapprox{n})\Delta t \right\rvert^2 \right).
\end{align}
From this equation,
we deduce that
\begin{align*}
    \mu_{\Delta t, N}(\qapprox{0}, \dots, \qapprox{N}) &= \mathscr{P}(\qapprox{0}) \prod_{n=0}^{N-1} \proba(\qapprox{n+1} | \qapprox{n}) \\
 &= \frac{\mathscr{P}(\qapprox{0})}{(2\pi \Delta t \sigma^2)^{n/2}}\exp\left(-\frac{1}{2\Delta t\sigma^2}\sum_{n=0}^{N-1}\left\lvert \qapprox{n+1} - \qapprox{n} - b(\qapprox{n})\Delta t \right\rvert^2 \right).
\end{align*}
A similar expression holds for~$\mu^\alpha_{\Delta t,N}$,
with $b^{\alpha}$ instead of $b$.
Given~\cref{assump:mes_inv_ci},
it follows that
\begin{align*}
&M_N^\alpha(\qapproxa{0}, \dots, \qapproxa{N})  \\
&= \frac{\mu_{\Delta t,N}(\qapproxa{0}, \dots, \qapproxa{N})}{\mu^\alpha_{\Delta t,N}(\qapproxa{0}, \dots, \qapproxa{N})}
=\frac{\displaystyle \exp\left(-\frac{1}{2\Delta t\sigma^2}\sum_{n=0}^{N-1}\left\lvert \qapproxa{n+1}-\qapproxa{n}-b(\qapproxa{n})\Delta t \right\rvert^2 \right)}{\displaystyle \exp\left(-\frac{1}{2\Delta t \sigma^2}\sum_{n=0}^{N-1}\left\lvert \qapproxa{n+1}-\qapproxa{n}-b^\alpha(\qapproxa{n})\Delta t \right\rvert^2 \right)}
\\ & =  \exp\left( \sum_{n=0}^{N-1}\frac{(\qapproxa{n+1}-\qapproxa{n})\cdot\left(b(\qapproxa{n})-b^\alpha(\qapproxa{n})\right)}{\sigma^2}\right) \!\exp\left(-\sum_{n=0}^{N-1}\frac{\left(|b(\qapproxa{n})|^2-|b^\alpha(\qapproxa{n})|^2\right) \Delta t}{2\sigma^2}\right).
\end{align*}
Writing $\qapproxa{n+1} - \qapproxa{n} = b^\alpha(\qapproxa{n}) \Delta t + \sigma \Delta W_n$,
we obtain
\begin{align*}
    &\frac{(\qapproxa{n+1}-\qapproxa{n})\cdot \bigl(b(\qapproxa{n})- b^\alpha(\qapproxa{n})\bigr)}{\sigma^2} \\
    &\qquad\qquad\qquad\qquad= \frac{b(\qapproxa{n})\cdot b^\alpha(\qapproxa{n})\Delta t - |b^\alpha(\qapproxa{n})|^2\Delta t + \sigma \bigl(b(\qapproxa{n})-b^\alpha(\qapproxa{n})\bigr)\cdot \Delta W_n}{\sigma^2},
\end{align*}
and therefore
\[
    M_N^\alpha(\qapproxa{0}, \dots, \qapproxa{N}) = \exp\left( \sum_{n=0}^{N-1}\frac{\bigl(b(\qapproxa{n})-b^\alpha(\qapproxa{n})\bigr)\cdot\Delta W_n}{\sigma} -\sum_{n=0}^{N-1}\frac{\left\lvert b(\qapproxa{n})-b^\alpha(\qapproxa{n})\right\rvert^2 \Delta t}{2\sigma^2}\right).
\]
Finally, since $b^{\alpha} = b +\alpha \sigma u$,
we obtain that
\begin{align}
    \label{eq:discrete_girsanov}
    M_N(\qapproxa{0}, \dots, \qapproxa{N}) = \exp\left(-\alpha \sum_{n=0}^{N-1}u(\qapproxa{n})^\top\Delta W_n -\frac{\alpha^2}{2}\sum_{n=0}^{N-1} |u(\qapproxa{n})|^2 \Delta t\right).
\end{align}
The above expression gives us the discrete Girsanov weight needed to reweight trajectories.
Note that~\eqref{eq:discrete_girsanov} is the discrete counterpart of~\eqref{eq:expression_expo}.
Thus, we have shown that
\begin{align}
    \label{eq:discrete_GK}
\expect\!\left[\sum_{n=0}^{N}R(\qapprox{n})S(\qapprox{0})\Delta t\right] = \expect\!\left[M_N(\qapproxa{0}, \dots, \qapproxa{N})\sum_{n=0}^{N}R(\qapproxa{n})S(\qapproxa{0})\Delta t\right],
\end{align}
which is the discrete counterpart of~\eqref{girsanov-application}.
\subsection{Numerical results}
\label{sec:numerical_experiments} 
We illustrate here the theoretical results established in~\cref{sec:asymptotic}. 
First, in~\cref{ssub:num_1d}, we simulate a one-dimensional overdamped Langevin dynamics where the perturbation is a fraction of the energy function. Then, we simulate in~\cref{sec:2d_num} a two-dimensional overdamepd Langevin dynamics with the free energy as the perturbation.
\subsubsection{One dimensional overdamped Langevin}
\label{ssub:num_1d}
\paragraph{The model}
We first consider the one dimensional overdamped Langevin dynamics
\begin{align}
    \label{eq:ovd_D}
    \d q_t = -V'(q_t)\,\d t + \sqrt{\frac{2}{\beta}}\,\d W_t,\qquad q_t \in \torus,
\end{align}
with~$\beta = 3$ and
\begin{align}
    \label{eq:potentiel_numD}
    V(q) = \frac{1 + \cos(2\pi q)}{2}.
\end{align}
We take the perturbation to be $u = -V'$, 
so the biased dynamics is
\begin{align*}
    \d q_t^\alpha = -\!\left(1+\alpha\sqrt{\frac{2}{\beta}}\right)V'(q_t^\alpha)\,\d t + \sqrt{\frac{2}{\beta}}\,\d W_t,\qquad q_t^\alpha \in \torus.
\end{align*}
For the numerical experiments we consider the schemes~\eqref{eq:disc_ref} and~\eqref{eq:disc_biased} for a time step of~$\Delta t  = 10^{-4}$ and a final time~$T$ in the range $[0, 2]$. 
Expectations were estimated using~$J = 10^{6}$ realizations.
The observables are~$R=V'$ and~$S = \beta R$, 
which appear in the computation of mobility, see for instance~\cite[Theorem 1]{FHS}.
\paragraph{Asymptotic scalings}
We first illustrate in~\cref{fig:fig1D} the behavior of the derivatives~$F_T'(0)$ and $F_T''(0)$ with respect to~$T$.
The dependence is clearly linear for the former and quadratic for the latter,
in very good agreement with the theoretical results~\eqref{eq:estimate_num} and~\eqref{eq:equivalent_second_derivative}.
Note that the curves in~\cref{fig:fig1D} are plotted with a 95\% confidence interval, 
which is nevertheless small and therefore difficult to distinguish on the figures.
Futhermore, we compute~$D_1$ and~$D_2$ from~\eqref{eq:estimate_num} and~\eqref{eq:equivalent_second_derivative}.
To this end we solve various Poisson equations. 
In the one dimensional case, the unique solution~$\Phi \in L^2_0(\mu)$ to the Poisson equation~$-\mathcal L  \Phi = \varphi$, with~$\varphi \in L^2_0(\mu)$, can be analytically obtained by~$\Phi = \Pi \widetilde{\Phi}$ with
\begin{align}
    \label{eq:sol_poiss_analytical}
    \widetilde{\Phi}(x,y) = \int_{0}^{x} \ee^{\beta V(u)}\left(\kappa-\beta\int_{0}^{u}\ee^{-\beta V(s)}\varphi(s)\,\d s\right)\d u ,
\end{align}
where
\begin{subequations}
\begin{align}
    \kappa & = \frac{\displaystyle\beta\int_0^1 \int_{0}^{u}\ee^{-\beta V(s)}\varphi(s)\,\d s\,\d u}{\displaystyle\int_0^1 \ee^{\beta V(u)}\,\d u}\cdot
\end{align}
\end{subequations}
This expression of~$\kappa$ ensures that~$\Phi$ is periodic 
and~$\expect_\mu[\Phi]=0$. 
Using the expression~\eqref{eq:sol_poiss_analytical}, we obtain~$D_1 \simeq 11.6$ and~$D_2 \simeq  116.6$ while the fits in~\cref{fig:fig1D} gives~$D_1 \simeq  11.4$ and~$D_2 \simeq  117.9$ which respectively represents an error of~$1.7\%$ and~$1.1\%$. 
\begin{figure}[h!]
    \centering
    \begin{subfigure}{0.49\textwidth}
        \centering
        \includegraphics[width=\linewidth]{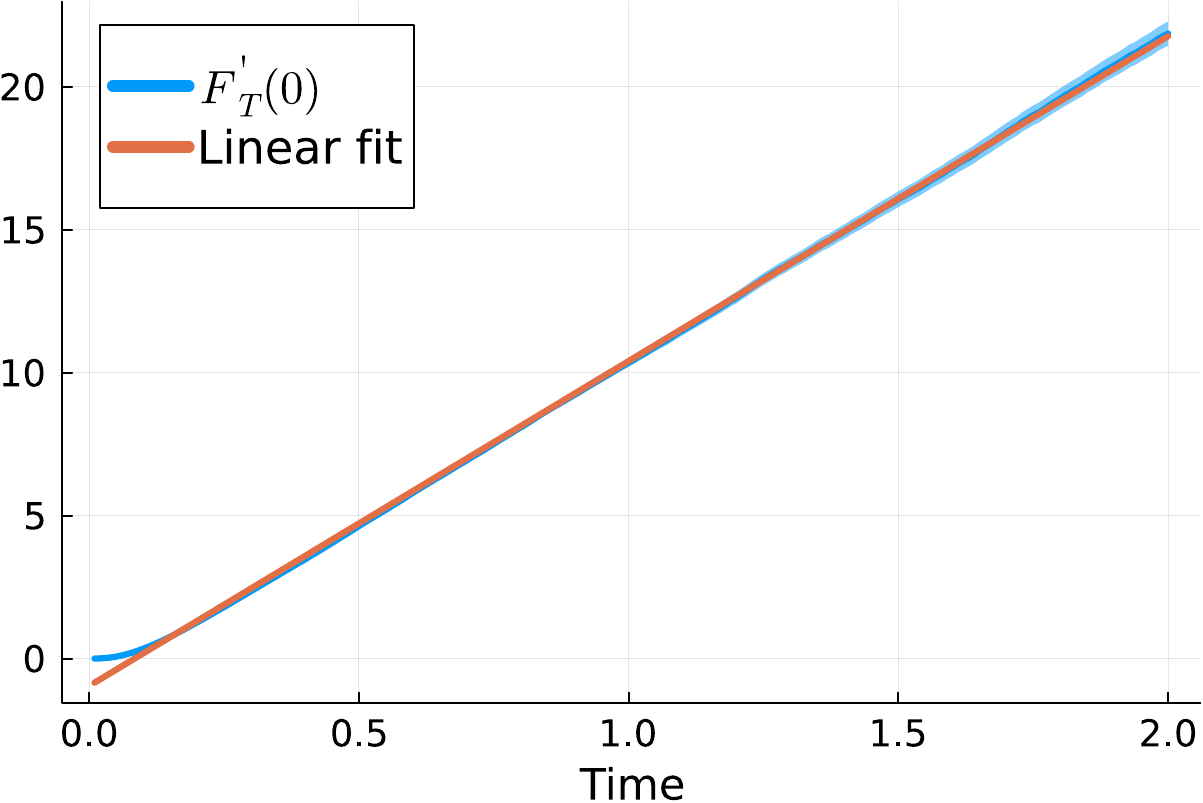}
        \caption{Time behavior of~$F_T'(0)$.}
        \label{fig:num1D}
    \end{subfigure}
    \hfill
    \begin{subfigure}{0.49\textwidth}
        \centering
        \includegraphics[width=\linewidth]{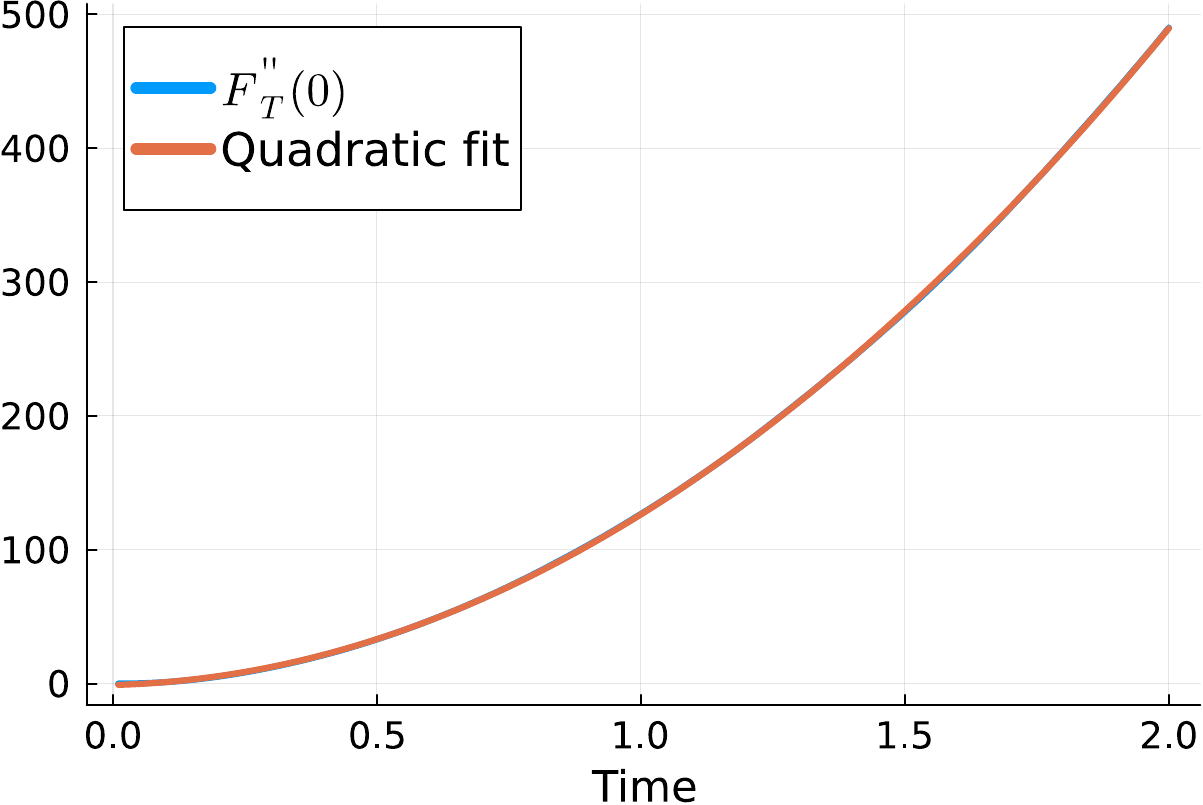}
        \caption{Time behavior of~$F_T''(0)$.}
        \label{fig:den1D}
    \end{subfigure}
    \caption{Time behavior of~$F_T'(0)$ and~$F_T''(0)$.}
    \label{fig:fig1D}
\end{figure}
\paragraph{Obtaining~$\alpha_T^*$}
In~\cref{fig:Obtaining_alpha}, 
we compare~$F_T(\alpha)$ with its quadratic approximation around~$0$ given by~\eqref{eq:taylor_expansion}. 
Furthermore, 
we numerically perform a quadratic fit of~$F_T$ around~$\widehat{\alpha}_T^*$.
The latter value is obtained with~\eqref{estimator-alpha} (which is the minimum of the analytic quadratic approximation~\eqref{eq:taylor_expansion}) to obtain a value of~$\alpha_T^*$.
\begin{figure}[h!]
    \includegraphics[width=0.85\textwidth]{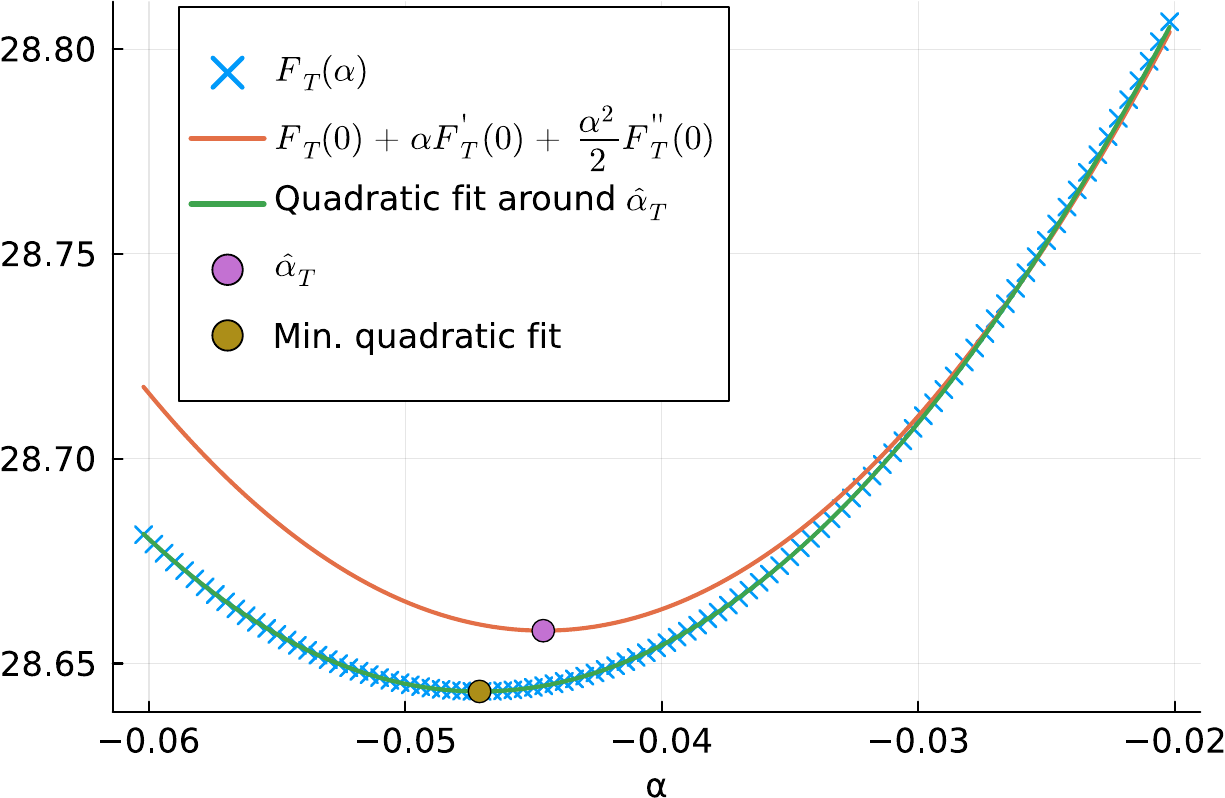}
    \caption{Comparison between~$F_T(\alpha)$,~\eqref{eq:taylor_expansion} and a quadratic fit around~$\widehat{\alpha}_T^*$ at~$T=2.0$.}
    \label{fig:Obtaining_alpha}
\end{figure}
\paragraph{Comparison of~$\alpha_T^*$ and~$\widehat\alpha_T^*$}
We compare the value of~$\widehat\alpha_T^*$ with~$\alpha_T^*$ obtained as described above. 
In~\cref{fig:true_scaling_alpha_1D}, 
the left panel~\cref{fig:comp_alpha} shows that these two values behave similarly with~$T$.
In~\cref{fig:scaling_log_alpha_1D} we show the asymptotic behavior in a~$\log$-$\log$ scale for~$T = 1.5$ to~$T = 2$ which is in agreement with the expected~$T^{-1}$ scaling~\eqref{eq:equi_alpha_opt}. 
In the following we only consider~$\widehat\alpha_T^*$ since~$\alpha_T^*$ is much more computationally expensive for only a small improvement in terms of variance reduction.
\begin{figure}[h!]
    \centering
    \begin{subfigure}{0.49\textwidth}
        \centering
        \includegraphics[width=\linewidth]{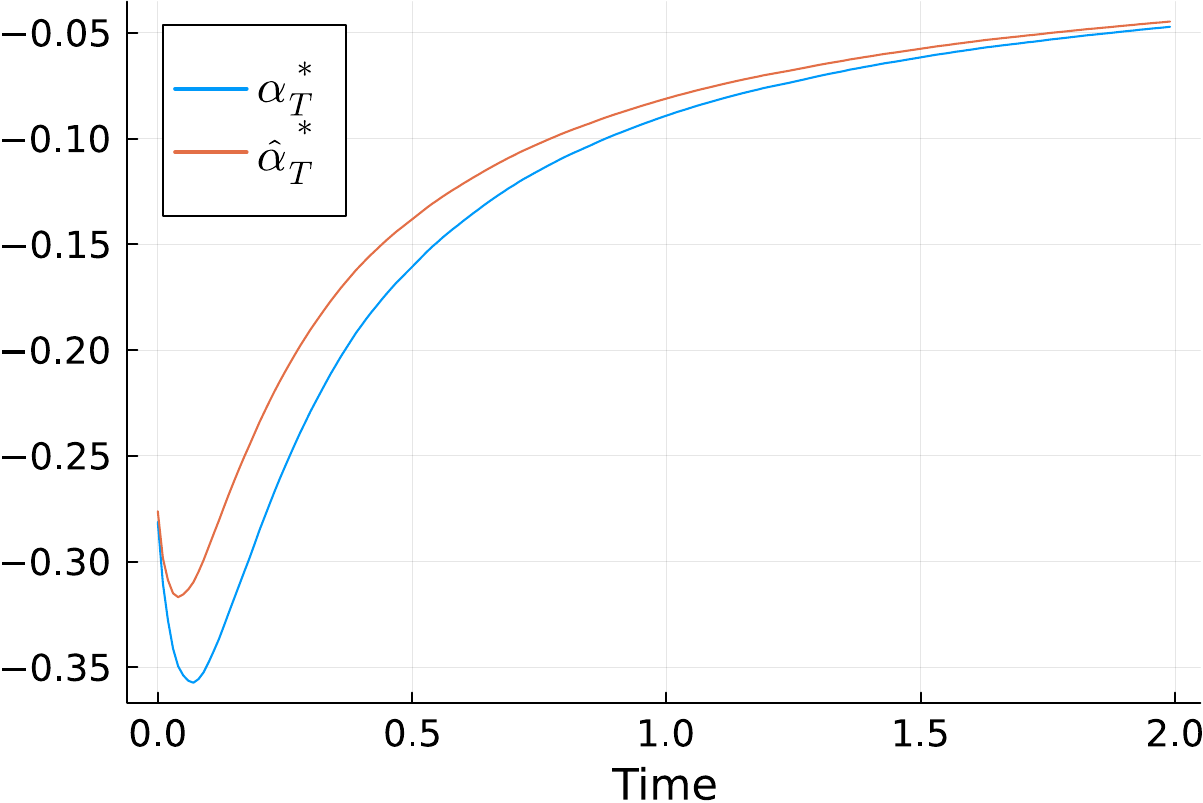}
        \caption{Comparison of~$\alpha_T^*$, obtained by quadratic interpolation, and~$\widehat\alpha_T^*$.}
        \label{fig:comp_alpha}
    \end{subfigure}
    \hfill
    \begin{subfigure}{0.49\textwidth}
        \centering
        \includegraphics[width=\linewidth]{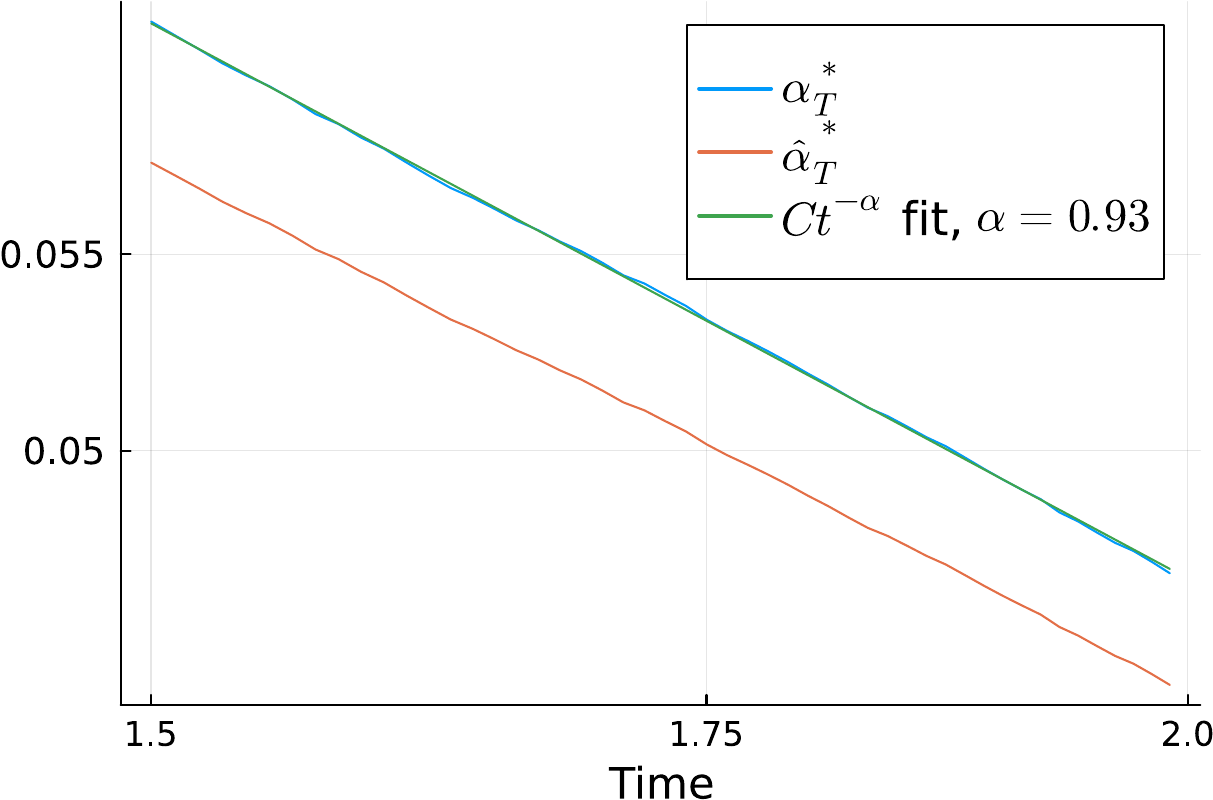}
        \caption{Scaling~$\log$-$\log$ of~$\alpha_T^*$ and~$\widehat\alpha_T^*$
        }
        \label{fig:scaling_log_alpha_1D}
    \end{subfigure}
    \caption{Comparison of~$\alpha_T^*$ and~$\widehat\alpha_T^*$ with respect to time.}
    \label{fig:true_scaling_alpha_1D}
\end{figure}
\paragraph{Variance reduction for~$\widehat{\alpha}_T^*$}
For~$\beta = 3$
we obtain a variance reduction of approximately $7\%$ for a time $T = 0.2$ and~$\widehat \alpha_T^* = -0.240$; 
see~\cref{fig:redvar1D}. 
The time~$T$ is close to the time when the Green-Kubo integral
\begin{align}
    \label{eq:1d_gk}
    \rho_T = \beta\int_{0}^{T}\expect\!\left[V'(q_t)V'(q_0)\right]\,\d t,
\end{align}
is close to its asymptotic limit; 
see~\cref{fig:gk1}.
We also report the variance reduction and optimal values estimators of the forcing~$\widehat{\alpha}_T^*$ for smaller values of~$\beta$ than~$3$.
The first panel of~\cref{fig:redvar1D} shows that in this case,
the larger the value of~$\beta$ (i.e the more the dynamics is metastable), 
the greater the variance reduction.
\begin{figure}[h!]
    \centering
    \begin{subfigure}{0.49\textwidth}
        \centering
        \includegraphics[width=\linewidth]{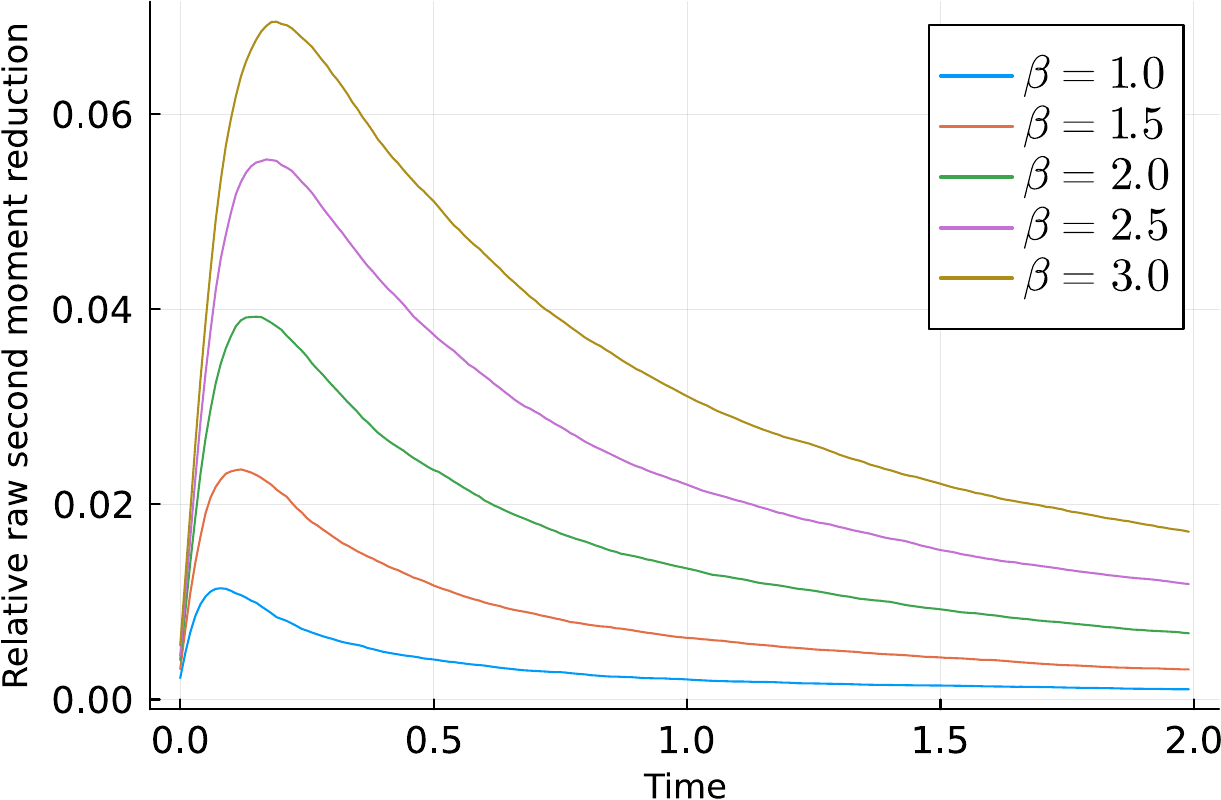}
        \caption{Relative reduction of the raw second moment with respect to~$T$.}
        \label{fig:var_redF_1D}
    \end{subfigure}
    \hfill
    \begin{subfigure}{0.49\textwidth}
        \centering
        \includegraphics[width=\linewidth]{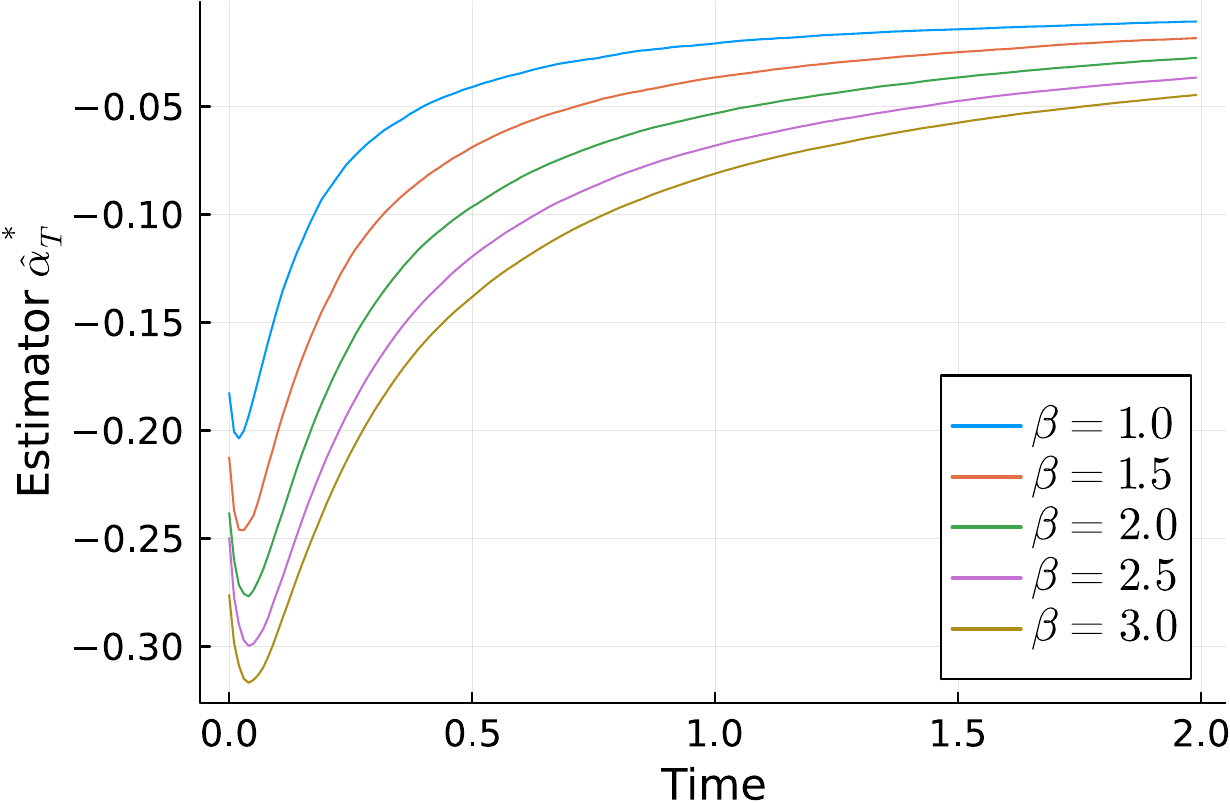}
        \caption{%
            Evolution of~$\widehat{\alpha}^*_T$ with respect to the truncation time~$T$,
            for various values of~$\beta$.
        }
        \label{fig:scaling_alpha_1D}
    \end{subfigure}
    \begin{subfigure}{0.49\textwidth}
        \centering
            \includegraphics[width=\linewidth]{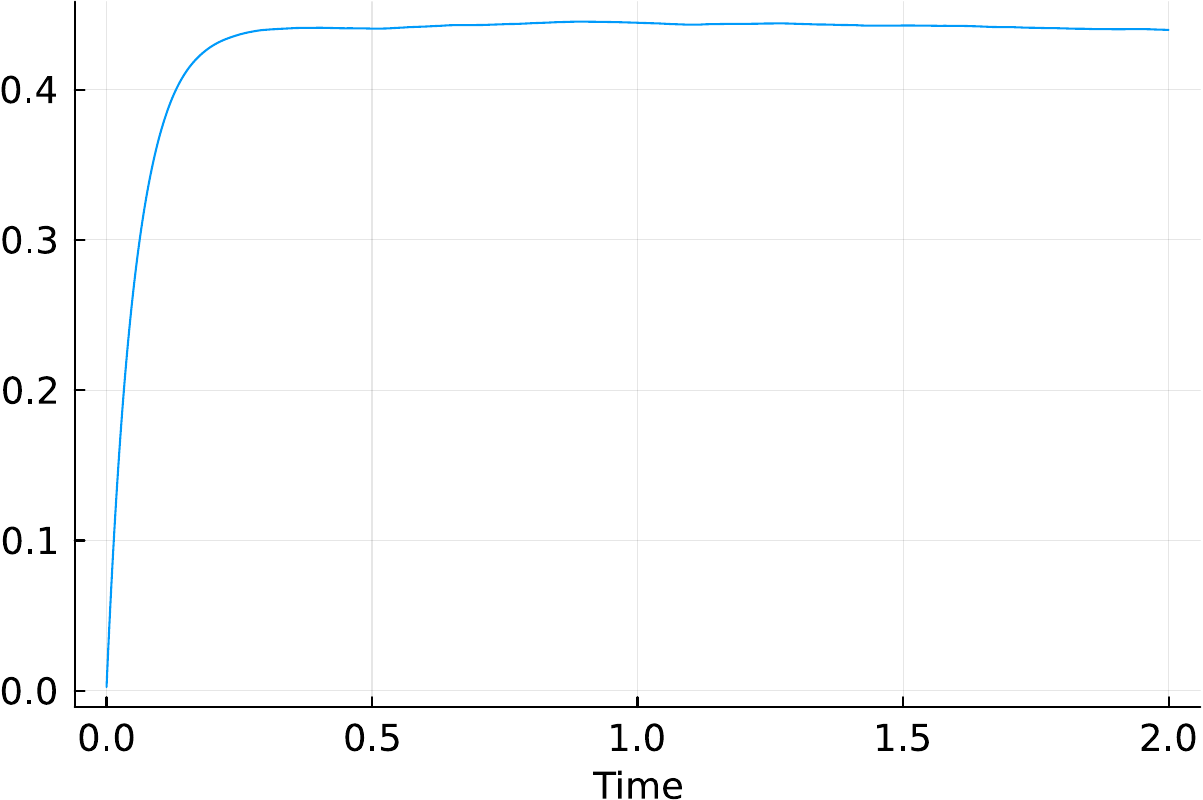}
            \caption{Green--Kubo estimator~\eqref{eq:1d_gk}.}
            \label{fig:gk1}
    \end{subfigure}
    \hfill
    \caption{Numerical results for the one-dimensional dynamics considered in~\cref{ssub:num_1d}.}
    \label{fig:redvar1D}
\end{figure}
\subsubsection{Two dimensional overdamped Langevin}
\label{sec:2d_num}
\paragraph{The model}
We present here a numerical experiment for the overdamped Langevin dynamics in dimension two:
\begin{align}
    \label{eq:ovd_2D}
    \d q_t = -\nabla V(q_t)\,\d t + \sqrt{\frac{2}{\beta}}\,\d W_t,\qquad q_t = (x_t,y_t)\in \torus^2.
\end{align}
Throughout this section we consider
\begin{align}
    \label{eq:potentiel_num2D}
    V(x,y) = \sin(4\pi (x+0.15))\left(2+\frac23\sin(2\pi (x+0.15))\right) + 4\cos(2\pi y) =: A(x) + B(y).
\end{align}
Unless otherwise stated, we fix $\beta = 2$.
Note that the potential~$V$ is separable.
\Cref{fig:trajectory} depicts the level sets of the potential~\eqref{eq:potentiel_num2D},
together with a realization of the dynamics~\eqref{eq:first_derivative},
simulated with the Euler--Maruyama discretization~\eqref{eq:disc_ref}.
This potential has 2 wells of different depths,
so that the dynamics~\eqref{eq:ovd_2D} is metastable.
Both~\cref{fig:trajectory,fig:xcoordinate}
(the latter illustrating only the evolution of the~$x$ coordinate)
show that the particle remains stuck in a well for a long time, before eventually jumping to the other well.
For the following numerical experiment,
the initial condition $q_0$ is sampled from the invariant probability measure with density proportional to~$\ee^{-\beta V(q)}$
by means of the rejection algorithm with uniform proposal.
The biasing drift is taken as~$u = -F'e_1$,
with~$e_1 = (1, 0)^\top$ and where $F$ denotes the free energy associated with the~$x$ coordinate  
(see~\cite{stoltz2010free}):
\begin{align*}
    F(x) &= \displaystyle - \frac{1}{\beta}\ln \left(\frac{1}{Z_\beta}\displaystyle\int_\torus
\e^{-\beta V(x,y)}
\d y\right) = A(x) - \frac{1}{\beta}\ln\!\left(\frac{1}{ Z_{\beta}} \int_\torus
\e^{-\beta B(y)}\d y\right),
\end{align*}
where~$Z_{\beta} = \displaystyle\int_{\torus^2}\e^{-\beta V(x,y)}\d x \, \d y$.
Therefore, $F' = A'$ and the biased dynamics reads
\begin{align}
    \label{eq:biased_dynamics_num}
    \d q_t^\alpha = -\left(\nabla V(q_t^\alpha) + \alpha\sqrt{\frac{2}{\beta}} F'(q_t^\alpha)e_1\right)\,\d t + \sqrt{\frac{2}{\beta}}\,\d W_t.
\end{align}
We consider the computation of the coefficient
 \begin{align}
    \label{eq:coeff_num}
    \rho = \int_0^\infty \expect_\mu\!\bigl[R(q_t)S(q_0)\bigr]\, \d t,
    \qquad S(q)=\beta R(q) = \beta\Bigl(\sin(2\pi x) - \expect_\mu\bigl[\sin(2\pi x)\bigr]\Bigr).
 \end{align}
 The chosen observable $R$ is smooth on the torus, 
 belongs to~$L^2_0(\mu)$ and takes different values in the two wells.
 The factor~$\beta$ in front of~$S$ in~\eqref{eq:coeff_num} is a reference to the previous 1-dimensional case for mobility and to the fact that in the general framework,~$S$ is a conjugated response function, 
 see~\cite[Section 5]{LS}.
 The constant~$\expect_\mu\bigl[\sin(2\pi x)\bigr]$ is pre-calculated with a deterministic quadrature method.
 The simulations were performed for $T = 200$,
$\Delta t = 4\times 10^{-4}$ and a number~$J = 10^8$ of realizations for the estimation of expectations.
 \begin{figure}[h!]
    \centering
    \begin{subfigure}{0.49\textwidth}
        \centering
        \includegraphics[width=\linewidth]{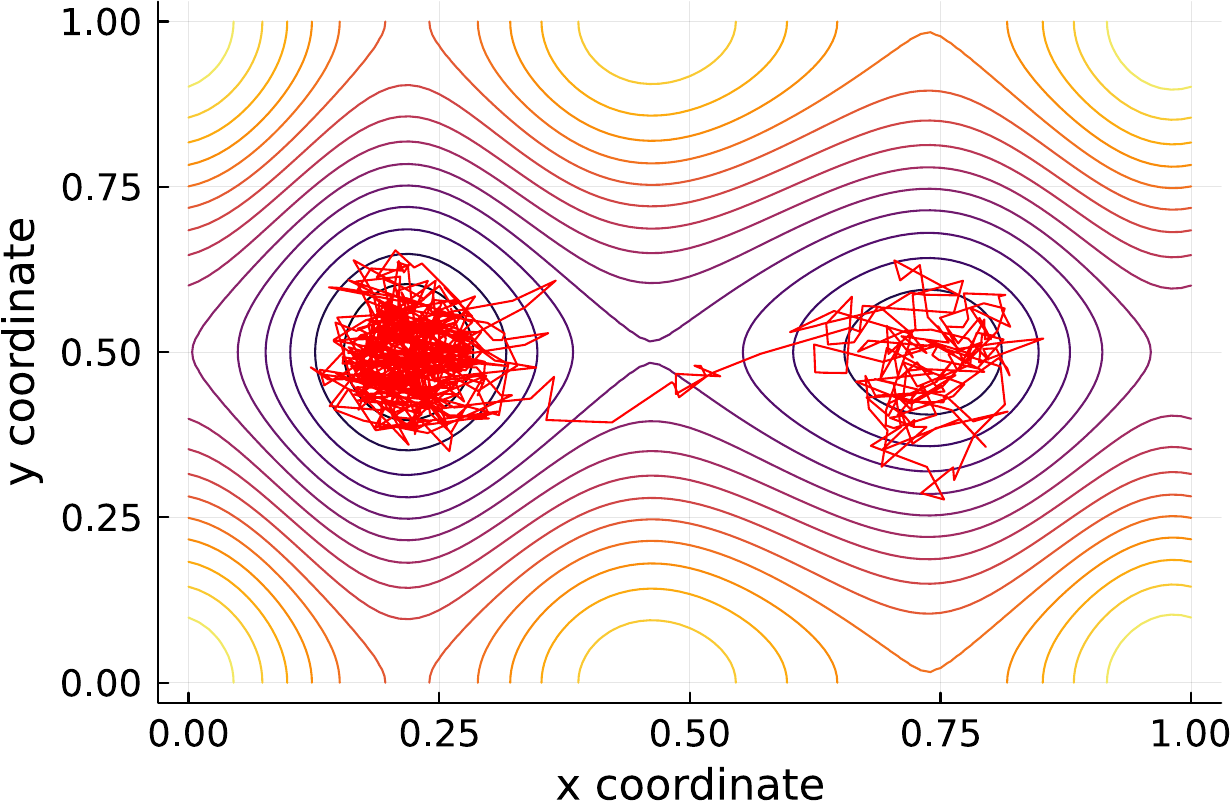}
        \caption{Level sets of the potential~\eqref{eq:potentiel_num2D} with a trajectory until~$T=1$.}
        \label{fig:trajectory}
    \end{subfigure}
    \hfill
    \begin{subfigure}{0.49\textwidth}
        \centering
        \includegraphics[width=\linewidth]{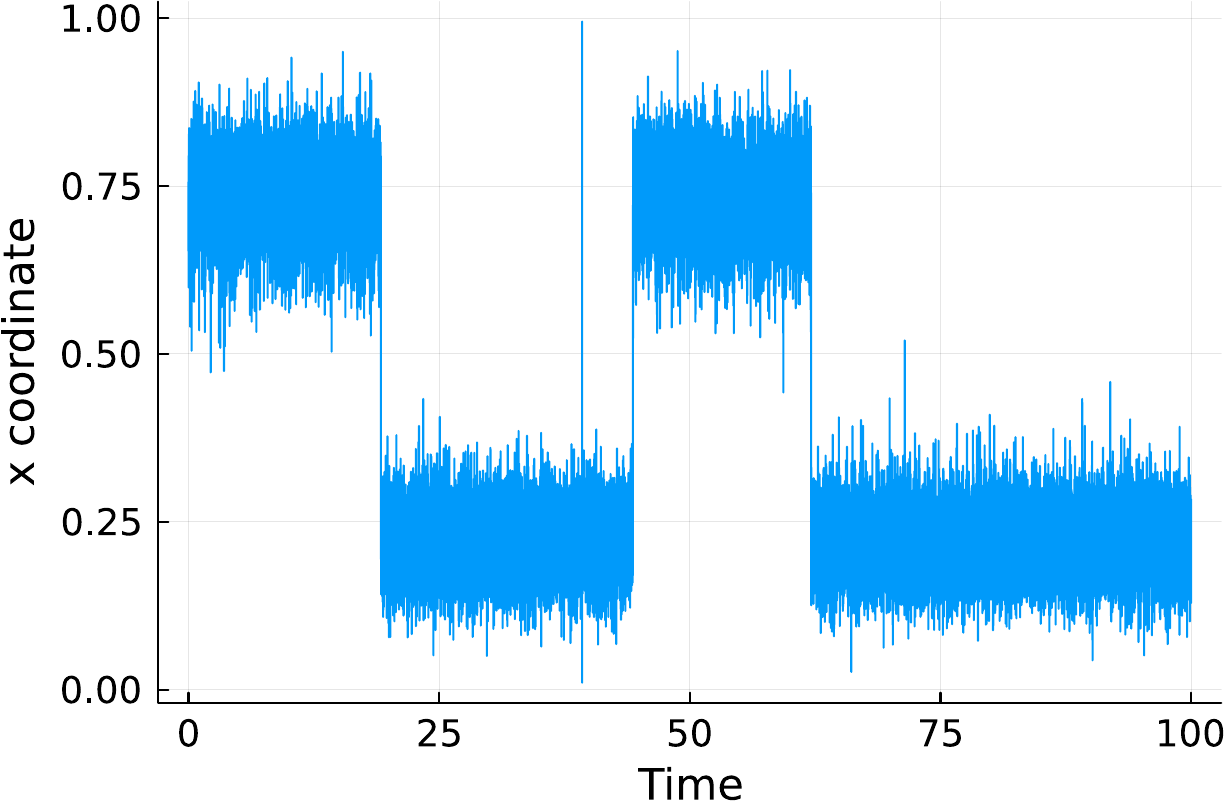}
        \caption{Projected trajectory in the $x$ variable for $\Delta t = 10^{-3}$ and $\beta = 2$.}
        \label{fig:xcoordinate}
    \end{subfigure}
    \caption{Visualization of the level sets of the potential~$V$ and trajectories of the reference dynamics~\eqref{eq:ovd_2D}.}
    \label{fig:all_figures}
\end{figure}
\paragraph{Motivation}
The impact of the parameter~$\alpha$ on the decay rate of the autocorrelation in~\eqref{eq:coeff_num} can be observed in~\cref{fig:multiples_autocor}, 
where we report~$\expect\!\left[R(q_t^\alpha)S(q_0^\alpha)\right]$ for various values of~$\alpha$.
We observe what can be physically expected:
for~$\alpha<0$,
increasing~$|\alpha|$ leads to faster decay of the correlation to~$0$,
while for~$\alpha > 0$, 
increasing~$|\alpha|$ leads to slower decay.
\begin{figure}
    \centering
        \includegraphics[height=6cm]{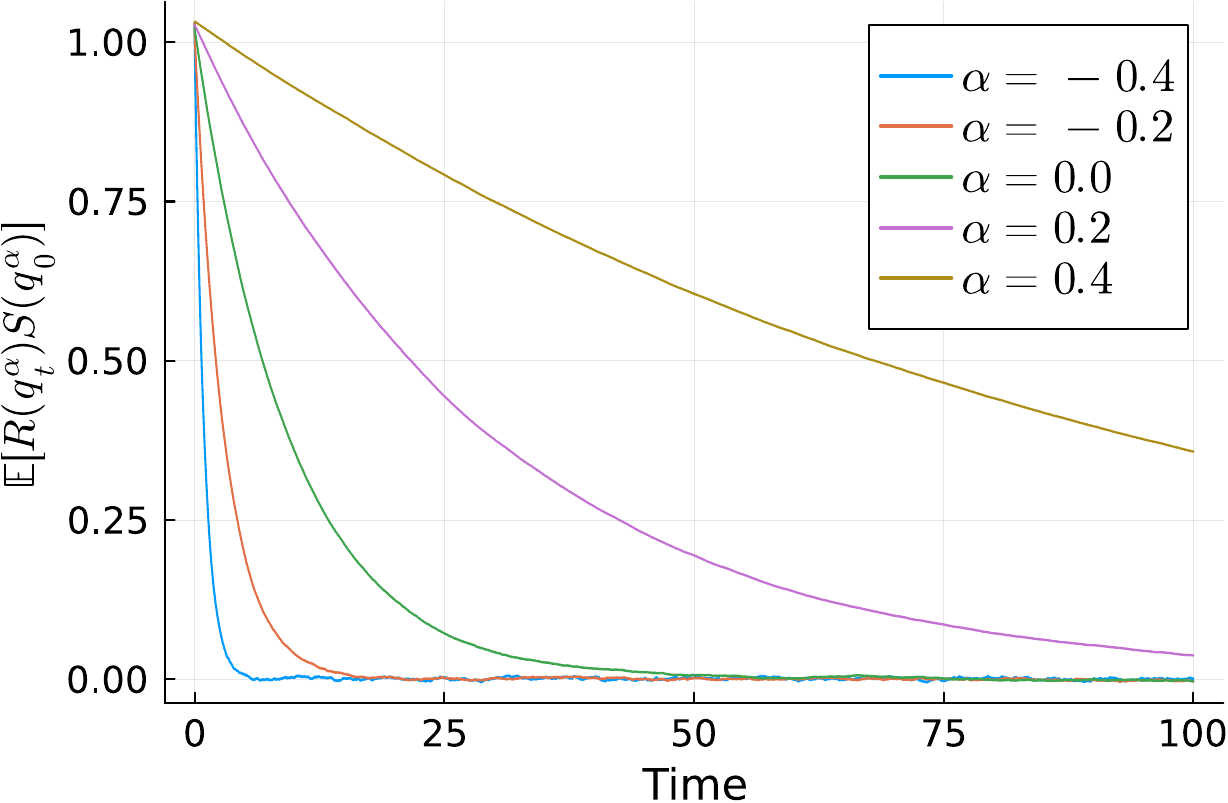}
        \caption{Evolution of~$\expect\!\left[R(q_t^\alpha)S(q_0^\alpha)\right]$, as a function of~$t$, for various values of~$\alpha$.}
        \label{fig:multiples_autocor}
\end{figure}
\paragraph{Scalings}
We next illustrate the asymptotic behavior of~$F_T'(0)$ and $F_T''(0)$ in~\cref{fig:f_prime_glob,fig:f_sec_glob}.
We observe that $F_T'(0)$ and $F_T''(0)$ are indeed asymptotically proportional to~$T$ and $T^2$, 
as predicted by~\cref{lemmaN,lemmaD} respectively.
In order to confirm these scalings we estimate the coefficients~$D_1$ and~$D_2$ by solving Poisson equations similarly to what was done in the one dimensional case, based on the expression~\eqref{eq:sol_poiss_analytical} and remplacing~$V$ by~$A$ (this is indeed possible since the potential~\eqref{eq:potentiel_num2D} is separable).
We obtain~$D_1 \simeq -67.2$ and~$D_2 \simeq 6.78\times 10^3$.
The numerical simulation gives~$D_1 \simeq -67.7$ and~$D_2 \simeq 7.32\times 10^3$, which respectively correspond to
a relative error compared to the analytical solution of~$0.74\%$ and~$8.0\%$.
\begin{figure}[h!]
    \centering 
    % Première sous-figure
    \begin{subfigure}{0.49\textwidth}
        \centering
        \includegraphics[width=\linewidth]{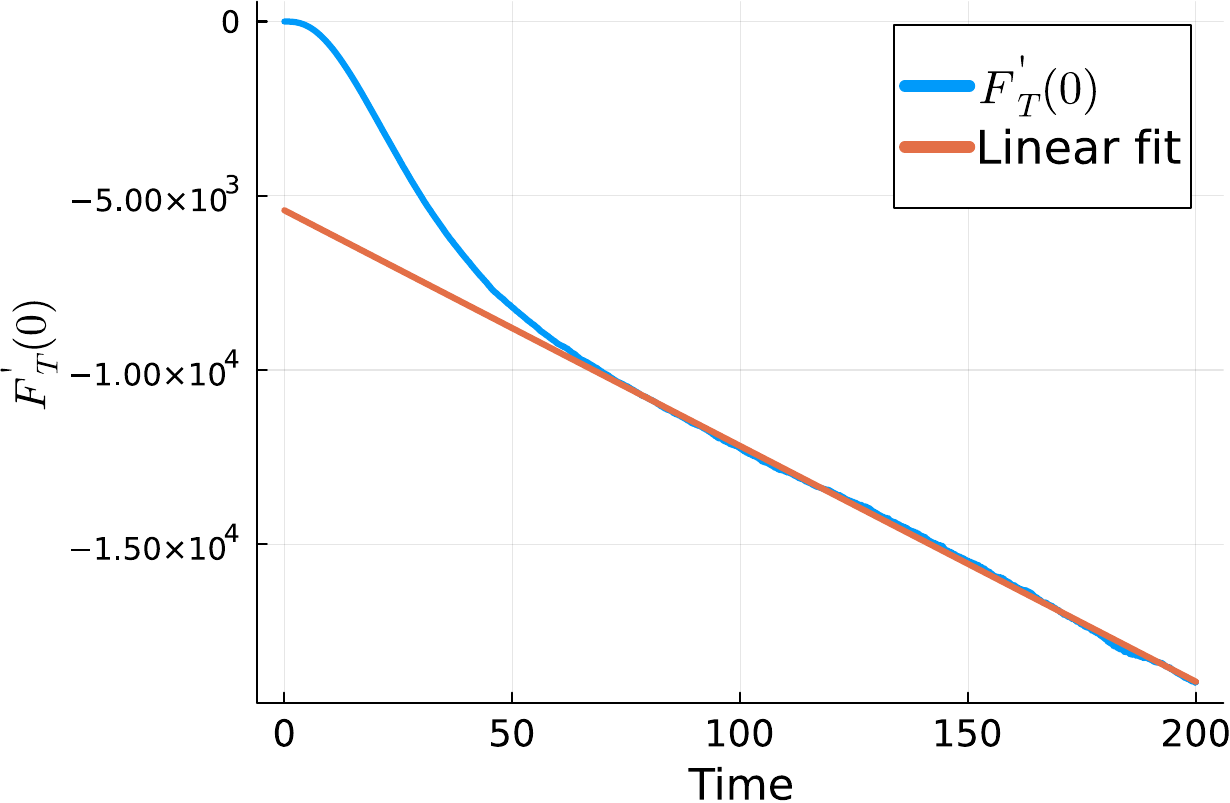}
        \caption{Time behavior of~$F_T'(0)$ and affine fit.}
        \label{fig:f_prime_glob}
    \end{subfigure}%
    \hfill
    \centering
    \begin{subfigure}{0.49\textwidth}
        \centering
        \includegraphics[width=\linewidth]{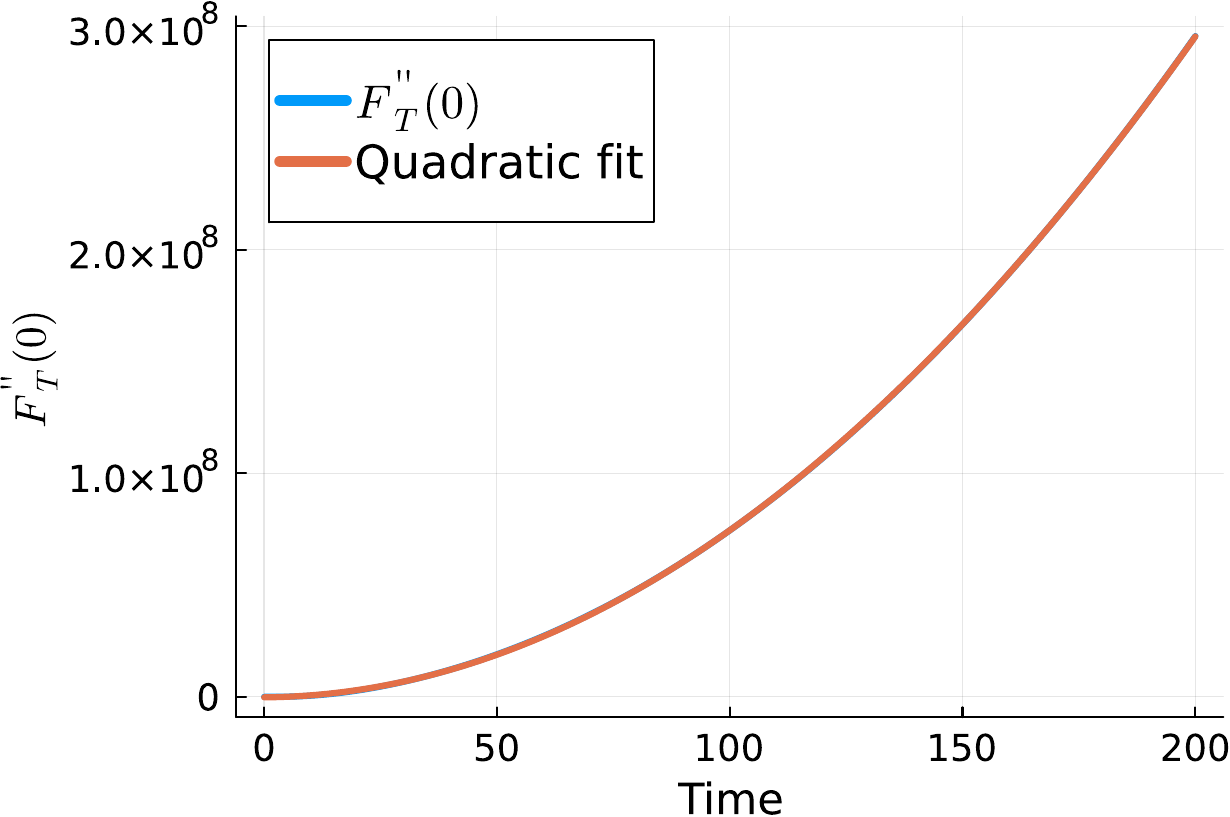}
        \caption{Time behavior of~$F_T''(0)$ with a polynomial fit of order~$2$.}
        \label{fig:f_sec_glob}
    \end{subfigure}%
    \caption{Time behavior of~$F_T'(0)$ and~$F_T''(0)$.}
    \label{fig:f_sec}
\end{figure}
\paragraph{Variance reduction and optimal~$\alpha$}
\Cref{fig:var_redF,fig:scaling_alpha} show the relative variance reduction of~$\displaystyle\frac{\left|F_T(0)- F_T(\widehat{\alpha}^*_T)\right|}{F_T(0)}$
and the behavior of~$\widehat{\alpha}^*_T$ for various values of~$\beta$.
For almost all values of~$\beta$ and $T$ considered, 
the relative reduction of~$F_T$ is less than~$1\%$.
One can observe that~\cref{fig:var_redF,fig:scaling_alpha} illustrate~\eqref{eq:gain_var_rel}, 
\textit{i.e.} that the lower the value of~$\widehat \alpha_T^*$ is,
the smaller the variance reduction is.
However,
according to~\cref{fig:multiples_autocor},
the autocorrelation function has not converged yet.
This example further demonstrates,
for a slightly more complex example than in~\cref{ssub:num_1d},
that the benefits in terms of variance reduction provided by the method studied in this paper are rather modest for large~$T$.
\begin{figure}[h!]
    \centering
    \begin{subfigure}{0.49\textwidth}
        \centering
        \includegraphics[width=\linewidth]{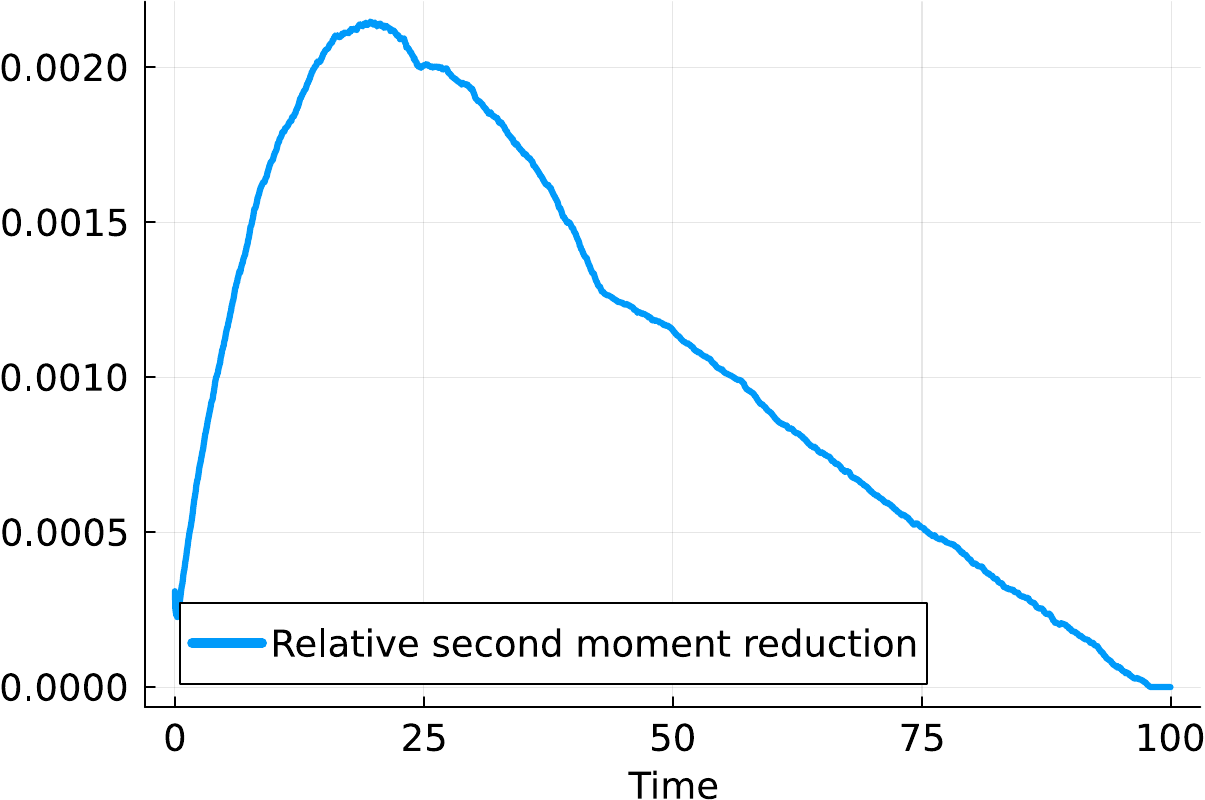}
        \caption{Relative reduction of the raw second moment with respect to~$T$.}
        \label{fig:var_redF}
    \end{subfigure}
    \hfill
    \begin{subfigure}{0.49\textwidth}
        \centering
        \includegraphics[width=\linewidth]{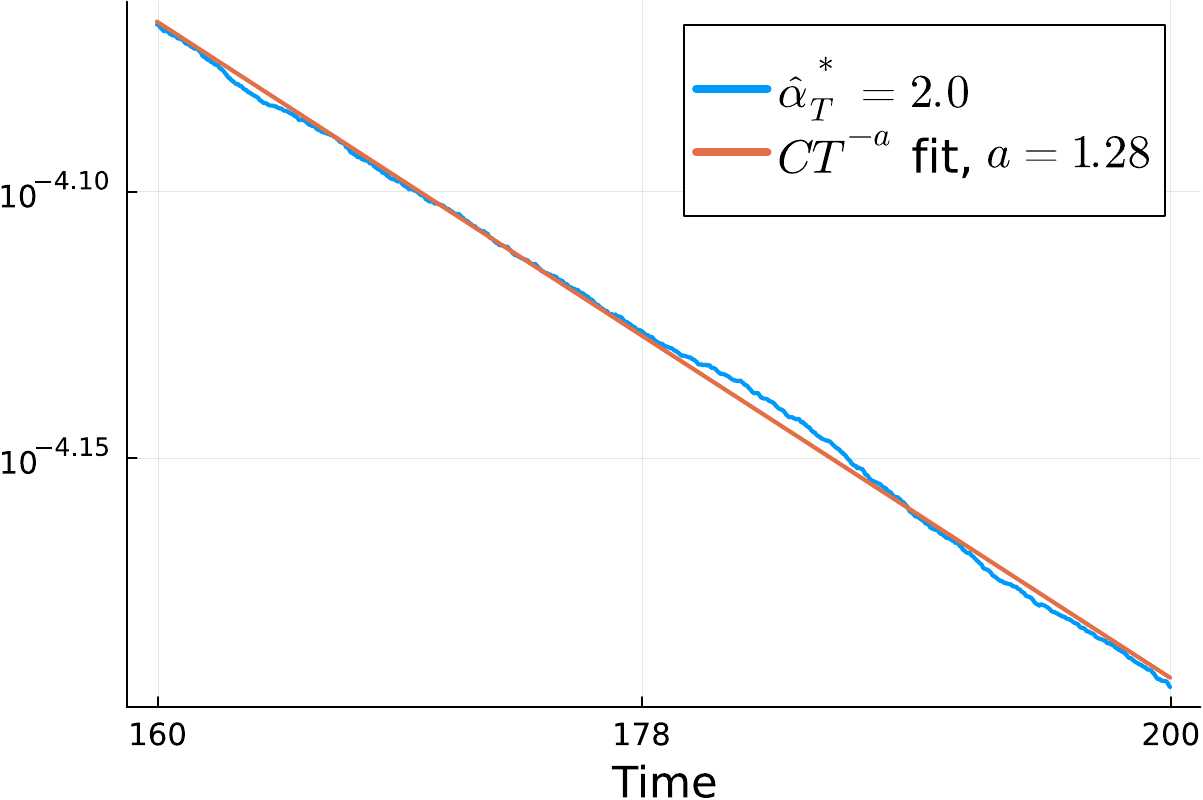}
        \caption{Scaling log-log of~$\widehat{\alpha}^*_T$.}
        \label{fig:scaling_alpha}
    \end{subfigure}
    \caption{Behavior of~$\widehat{\alpha}^*_T$.}
    \label{fig:alpha_behavior_2D}
\end{figure}
\section{Extensions and perspectives}
\label{sec:conclusion}
In this final section,
we list possible extensions of the variance reduction method for the calculation of transport coefficients presented in this work.
\begin{itemize}
    \item
        The first natural extension is to optimize the functional form of the biasing drift~$u$, 
        which amounts to optimizing the value of~$\frac{D_1^2}{D_2}$ in view of~\eqref{eq:conv_min} and~\eqref{eq:gain_var_rel}.
        In dimension one,
        it is natural to choose a perturbation of the form $u = -V'$.
        However,
        this approach is impractical in high-dimensional setting,
        such as for the simulation of proteins, as~$V'$ will be large, 
        and will depend on irrelevant degrees of freedom such as the solvent. 
        A biasing by a function depending only on a selected number of degrees of freedom (encoded by a collective variable),
        such as the free energy considered is~\cref{sec:2d_num}, 
        is more relevant.
    \item
        Throughout the theoretical part of this paper,
        we have restricted our attention to the continuous-time setting,
        but in practice the dynamics and estimators must be discretized, 
        as done in~\cite{DHK,donati2018girsanov}.
        Therefore, a natural direction for future work would be carry out the numerical analysis of the weighted Green--Kubo estimator,
        \textit{i.e.}\ after time discretization, 
        and rigorously quantify the time-step bias.
    \item
        Since, for a given $\alpha$, the stochastic process $\mathcal E_t(\alpha)$ in~\eqref{girsanov-application} is a martingale,
        we could also consider instead of~\eqref{eq:good_estimator} the estimator
        \[
            \widehat{B}_T^{\alpha} = \int_0^T R\left(q_t^\alpha\right) S\left(q_0^\alpha\right) \mathcal{E}_t(\alpha) \, \d t.
        \]
        Preliminary simulations suggest a variance reduction similar to that obtained earlier in this work. 
        However, 
        it is more difficult to perform the theoretical analysis developed here and thus to obtain precise analytical results.
    \item
        As another research direction,
        one could extend our results to the setting of degenerate diffusions,
        such as underdamped Langevin dynamics,
        which are often more relevant than overdamped Langevin dynamics in applications.
        Let us mention that path reweighting has already been employed in this context, see~\cite{kieninger2023girsanov}.
\end{itemize}
However, 
given the relatively modest variance reduction observed in the numerical experiments,
it is unclear whether such a research effort is warranted and so it seems more relevant to focus on other more effective methods such as fluctuation methods~\cite{10.1093/imanum/drac073} or methods based on Einstein's formula~\cite{pavliotis2024neuralnetworkapproachesvariance,bob}.  
\appendix
\section{Auxiliary results}
This section contains two useful auxiliary results.
The first one that follows concerns the expectation of the product of three Brownian integrals.
\begin{lemma}
    \label{lemme:egalite-moment-3}
    Let $T>0$,
    and suppose that $f$, $g$ and $h$ are smooth functions on the torus.
    Introduce the Brownian martingales
    \[
        I_T = \int_0^T f(q_t)\,\d W_t, \qquad  J_T = \int_0^T g(q_t)\,\d W_t, \qquad  K_T = \int_0^T h(q_t)\,\d W_t,
    \]
    with~$q_t$ a continuous stochastic process adapted to~$(W_t)_{t\geq0}$.
    Then, for any initial condition~$q_0\in\torus^d$,
    \[
        \expect^{q_0}\!\left[I_T\,J_T\,K_T \right] =
        \expect^{q_0}\left[ \langle J,K\rangle_T I_T\right] + \expect^{q_0}\left[ \langle I,K\rangle_T J_T\right] + \expect^{q_0}\left[ \langle I,J\rangle_T K_T\right],
    \]
    where $\langle \, \cdot \, , \, \cdot \, \rangle_T$ denotes the covariation process at time~$T$.
\end{lemma}
\begin{proof}
    Letting $\Psi(I_t,J_t,K_t) = I_t J_t K_t$ and applying It\^o's formula,
    we obtain that
    \begin{align*}
        \d \Psi (I_t,J_t,K_t) 
            =&\,
            J_t K_t f(q_t) \d W_t + I_t K_t g(q_t) \d W_t + I_t J_t h(q_t) \d W_t \\
            &+ I_t\langle J,K \rangle_t + J_t\langle I, K\rangle_t + K_t\langle I, J\rangle_t.
    \end{align*}
    By integrating in time,
    \begin{align*}
        \Psi_T
            &=
            \int_{0}^{T}
            \Bigl(
                I_t g(q_t) h(q_t)
                + J_t f(q_t) h(q_t)
                + K_t f(q_t) g(q_t)
            \Bigr)
            \, \d t
            \\
            &\qquad
            + \int_{0}^{T}
            \Bigl( J_t K_t f(q_t)
                + I_t K_t g(q_t)
            + I_t J_t h(q_t) \Bigr) \, \d W_t.
    \end{align*}
    Then,
    taking the expectation and using the martingale property of stochastic integrals,
    we have
    \begin{align*}
        \expect^{q_0} \!\left[\Psi_T\right]
            &=
            \int_{0}^{T}
            \Bigl(
                \expect^{q_0}\!\left[ I_t g(q_t) h(q_t) \right]
                + \expect^{q_0} \!\left[ J_t f(q_t) h(q_t) \right]
                + \expect^{q_0} \!\left[ K_t f(q_t) g(q_t) \right]
            \Bigr) \d t \\
            &=
            \expect^{q_0} \biggl[ I_T \int_{0}^{T} g(q_t) h(q_t) \, \d t \biggr]
            + \expect^{q_0} \biggl[ J_T \int_{0}^{T} f(q_t) h(q_t) \, \d t \biggr] \\
            &\quad + \expect^{q_0} \biggl[ K_T \int_{0}^{T} f(q_t) g(q_t) \, \d t \biggr],
    \end{align*}
    which concludes the proof.
\end{proof}
    The following result combines a central limit theorem for martingales with non-stationary increments,
    with a central limit theorem for additive functionals of Markov processes started at a point.
    For results of the former type, see~\cite{MR668684,MR2368952};
    and for results of the latter type,
    see~\cite[Section 8]{cattiaux}.
\begin{theorem}
    \label{corollary:central_limit}
    Suppose that~\cref{assumption:ex_mes_inv} is satisfied (recall in particular that $\sigma$ takes values in~$\real^{d \times m}$),
    and let~$(q^x_t)_{t \geq 0}$, 
    for~$x \in \torus^d$, 
    denote the solution to
    \[
        \d q^x_t = b (q^x_t) \, \d t + \sigma(q^x_t) \, \d W_t, \qquad q^x_0 = x.
    \]
    Consider smooth functions $g\colon \torus^d \to \real^n$
    and $h\colon \torus^d \to \real^m$ %for~$x \in \torus^d$
    and denote by $\psi \in H^1_0(\mu)^n$ denote the~$\real^n$-valued solution to the Poisson equation $- \mathcal L \psi = g$.
    Then
    \[\displaystyle
        \begin{bmatrix}
            \displaystyle\frac{1}{\sqrt{T}} \int_{0}^{T} g(q^x_t) \, \d t  \\[.3cm]
            \displaystyle\frac{1}{\sqrt{T}} \int_{0}^{T} h(q^x_t) \cdot \d W_t
        \end{bmatrix}
        \xrightarrow[T \to \infty]{\rm Law}
        \mathcal N\left(
            0,
            \Sigma \right),
    \]
    with
    \[
        \Sigma = \left[\begin{matrix}
            \expect_\mu\!\left[\nabla \psi \sigma \sigma^\top \nabla \psi^\top\right] & \expect_\mu\!\left[\nabla \psi \sigma h\right] \\
            \expect_\mu\!\left[h^\top \sigma^\top \nabla \psi^\top\right] & \expect_\mu\!\left[h^\top h\right]
        \end{matrix}\right] \in \real^{(n+1)\times (n+1)},
        \]
    where $\nabla \psi \colon \torus^d \to \real^{n \times d}$ is the Jacobian matrix of~$\psi$.
\end{theorem}
\begin{proof}
    Note first that, by It\^o's formula,
    \begin{align*}
        \begin{bmatrix}
            \displaystyle\frac{1}{\sqrt{T}} \int_{0}^{T} g(q^x_t) \, \d t  \\[.3cm]
            \displaystyle\frac{1}{\sqrt{T}} \int_{0}^{T} h (q^x_t) \cdot \d W_t
       \end{bmatrix}
       &=
       \begin{bmatrix}
        \displaystyle\frac{1}{\sqrt{T}} \int_{0}^{T} -\mathcal{L}\psi(q^x_t) \, \d t  \\
        \displaystyle \frac{1}{\sqrt{T}} \int_{0}^{T} h(q^x_t) \cdot \d W_t
      \end{bmatrix}
      \\
       &=
       \begin{bmatrix}
        \displaystyle\frac{1}{\sqrt{T}} \int_{0}^{T} \sigma^\top (q^x_t)\nabla\psi(q^x_t) \cdot \d W_t  \\
        \displaystyle \frac{1}{\sqrt{T}} \int_{0}^{T} h(q^x_t) \cdot \d W_t
       \end{bmatrix}
       -
       \frac{1}{\sqrt{T}}
       \begin{bmatrix}
           \psi(q_T) - \psi(q_0)\\
            0
       \end{bmatrix}.
    \end{align*}
    Since $\psi\colon \torus^d \to \real^{n}$ is bound,
    the second term tends to 0 almost surely.
    Thus, by Slutsky's theorem, it is sufficient to prove that
    \[
        \mathcal M(t) := \frac{1}{\sqrt{t}} \int_{0}^{t} v(q^x_s) \, \d W_s
        \xrightarrow[t \to \infty]{\rm Law}
        \mathcal N\left(
            0,
            \Sigma
        \right),
        \qquad
        v(q) :=
        \begin{bmatrix}
            \sigma^\top (q)\nabla\psi(q) \\
            h^\top(q)
       \end{bmatrix}
       \in \real^{(n+1) \times m}.
    \]
    To this end,
    denote by $(M_k)_{k \in \mathbb{N}}$ the following sequence of martingales:
    \[
        M_k(t) = \frac{1}{\sqrt{k}}\int_0^{kt}v(q_s)\,\d W_s.
    \]
    For any $k\in\mathbb{N}$,
    the matrix-valued quadratic variation~(see \cite[Definition 3.18]{pavliotis2008multiscale}) of $M_k(t)$ is given by
    \[
        \langle M_k \rangle_t
        = \frac{1}{k}\int_0^{kt}v(q_s) v(q_s)^\top \,\d s.
    \]
    Let $v^{i,j}(q) = v_i(q)\cdot v_j(q)$,
    where for $i \in \{1, \dotsc, n+1\}$
    the notation $v_i$ refers to the $i$-th row of $v$.
    Denote by~$\langle M_k \rangle_t^{i,j}$ the entry on row~$i$ and column~$j$ of $\langle M_k \rangle_t$.
    Then, writing $v^{i,j} = \Pi v^{i,j} + \expect_{\mu}\!\left[v^{i,j}\right]$,
    we obtain that
    \begin{align}
        \label{eq:quadra_covar}
        \langle M_k \rangle_t^{i,j} = \left( \expect_{\mu}\left[v^{i,j}\right]
        + \frac{1}{kt}\int_0^{kt}\Pi v^{i,j}(q_s)\,\d s \right) t.
    \end{align}
    The second term in the bracket on the right-hand side is an ergodic average of a mean-zero observable.
    By the usual reasoning involving a Poisson equation,
    see e.g.~\cite[Section 4.2]{MattinglyStuartTretyakov2010} or the proof of~\cref{proposition:asymptotic_variance_gk} here,
    it holds that
    \[
        \expect\!\left[\,\left|\frac{1}{kt}\int_0^{kt}\Pi v^{i,j}(q_s)\,\d s\right|^2\right] \xrightarrow[k \to \infty]{} 0.
    \]
    By continuity of the sum,
    and since convergence in squared expectation implies convergence in distribution,
    we conclude, 
    that for all $t \geq 0$,
    \begin{align}
        \label{lim:distrib_matrix}
        \langle M_k \rangle_t^{i,j}
        \xrightarrow[k \to \infty]{\rm Law} \expect_{\mu}\left[v^{i,j}\right]t,\qquad
        \langle M_k \rangle_t
        \xrightarrow[k \to \infty]{\rm Law} \Sigma t.
    \end{align}
    Since realizations of $(M_k)_{k\in\mathbb{N}}$ are almost surely continuous,
    they almost surely exhibit no jumps.
    Together with~\eqref{lim:distrib_matrix},
    this implies that the assumptions of~\cite[Theorem 2.1]{MR2368952} are satisfied,
    and so this theorem gives
    \[
        M_k \xrightarrow[k\to \infty]{\rm Law} \sqrt{\Sigma} B \qquad \text{with respect to $D^{n+1}$},
    \]
    where $B$ is a standard $\real^{n+1}$ Brownian motion and $D \equiv D([0, \infty))$ is the space of right-continuous real-valued functions on~$[0, \infty)$ endowed with the Skorohod $J_1$ topology.
    To conclude, we use the continuous mapping theorem~\cite[Theorem 2.7]{MR1700749}.
    For $t \geq 0$, let~$\mathcal P_t$ denote the pointwise evaluation functional
    \[
        \mathcal P_t \colon D^{n+1} \ni f \mapsto f(t) \in \real^{n+1}.
    \]
    It can be shown that this operator is continuous at any continuous $f$,
    see~\cite{2022arXiv221016026K}.
    Since Brownian motion is almost surely continuous,
    we deduce from the continuous mapping theorem that
    \begin{equation}
        \label{eq:pointwise_convergence}
        \mathcal P_t M_k
        \xrightarrow[k \to \infty]{\rm Law}
        \mathcal P_t \left(\sqrt{\Sigma} B\right) \sim \mathcal N(0, t\Sigma),
    \end{equation}
    This also follows from~\cite[Theorem~3.10.2]{MR0838085}.
    To conclude, we use that
    \begin{align}
        \notag
        \mathcal M(t)
        &= \frac{1}{\sqrt{t}} \int_{0}^{\lfloor t \rfloor} v(q^x_s) \, \d W_s
        +\frac{1}{\sqrt{t}} \int_{\lfloor t \rfloor}^{t} v(q^x_s) \, \d W_s \\
        \label{eq:two_terms_lemma}
        &= \sqrt{\frac{\lfloor t \rfloor}{t}} M_{\lfloor t \rfloor}(1)
        +\frac{1}{\sqrt{t}} \int_{\lfloor t \rfloor}^{t} v(q^x_s) \, \d W_s.
    \end{align}
    The first term in~\eqref{eq:two_terms_lemma} converges to~$\mathcal N(0, \Sigma)$ by~\eqref{eq:pointwise_convergence} with $t = 1$.
    For the second term in~\eqref{eq:two_terms_lemma}, we have by Itô's lemma that
    \begin{align*}
        \expect \left[ \left( \frac{1}{\sqrt{t}} \int_{\lfloor t \rfloor}^{t} v(q^x_s) \, \d W_s \right)^2 \right] = \frac{1}{t}\expect \left[ \int_{\lfloor t \rfloor}^{t} |v(q^x_s)|^2 \, \d s\right] \leq \frac{t - \lfloor t \rfloor}{t}\|v\|^2_{L^{\infty}},
    \end{align*}
    so this term converges to 0 as~$t \to \infty$~in squared expectation,
    and thus also in probability.
    The statement then follows from another application of Slutsky's theorem.
\end{proof}
\begin{lemma}
    \label{lem:min_comp}
    Let~$(f_n)_{n\in\mathbb{N}}$ be the sequence of functions defined by~\eqref{eq:suite_fn}. 
    Then there exists a compact set~$K\subset \real$ independent of~$n$ such that
    $$
    \left(\argmin_{\alpha \in \real} f_n(a)\right)_{n\in\mathbb{N}}\subset K.
    $$
\end{lemma}
\begin{proof}
    Defining $a_{\infty} = -\frac{D_1}{D_2}$,
    let us rewrite~$f_\infty$ as
    \[
        f_{\infty}(a) = \frac{D_2}{2}(a - a_{\infty})^2 + f^{\min}_{\infty}, \qquad f^{\rm min}_{\infty} = -\frac{D_1^2}{2D_2}.
    \]
    Fix $\delta > 0$, let $K_{\delta} = [a_{\infty} - \delta , a_{\infty} + \delta]$,
    and let $\varepsilon = \frac{D_2 \delta^2}{6}$.
    By uniform convergence on compact subsets,
    there exists $n_\varepsilon$ such that
    \begin{align*}
        \forall n \geq n_{\varepsilon}, \qquad
        \left\{
        \begin{aligned}
            &f_n(a_{\infty})   &\le f_{\infty}(a_{\infty}) + \varepsilon  = f^{\rm min}_{\infty} + \varepsilon = f^{\rm min}_{\infty} + \frac{D_2 \delta^2}{6},\\
            &f_n(a_{\infty}+\delta) &\ge f_{\infty}(a_{\infty}+\delta) - \varepsilon = \frac{D_2 \delta^2}{2} + f^{\rm min}_{\infty} - \varepsilon = f^{\rm min}_{\infty} + \frac{D_2 \delta^2}{3},\\
            &f_n(a_{\infty}-\delta) &\ge f_{\infty}(a_{\infty}-\delta) - \varepsilon = \frac{D_2 \delta^2}{2} + f^{\rm min}_{\infty} - \varepsilon = f^{\rm min}_{\infty} + \frac{D_2 \delta^2}{3}.
        \end{aligned}
        \right.
    \end{align*}
    Consequently, by convexity of $f_n$,
    it holds that~$\argmin_{\alpha \in \real} f_n(a) \in K_{\delta}$ for all $n \geq n_{\varepsilon}$.
    Since $\delta$ was arbitrary,
    this concludes the proof.
\end{proof}
\section*{Acknowledgments}
The authors thank Bettina Keller for stimulating discussions and Clément Guillot for helpful insights on the numerical part.
This project has received funding from the European Research Council (ERC) under the European Union’s Horizon 2020 research and innovation programme (project EMC2, grant agreement No 810367). 
We also acknowledge funding from the Agence Nationale de la Recherche, under grants ANR-21-CE40-0006 (SINEQ) and ANR-23-CE40-0027 (IPSO).
\ifsiamart
\bibliographystyle{siamplain}
\bibliography{main}
\else
\printbibliography
\fi
\end{document}